\documentclass[twoside, 14pt]{article}
\usepackage{amsmath,amssymb,amsthm,mathrsfs}
\usepackage{ulem}
\usepackage[margin=2.0cm]{geometry} 
\usepackage{lipsum}
\usepackage{titlesec,hyperref}
\usepackage{fancyhdr}
\usepackage[numbers,sort&compress]{natbib}
\usepackage{color}
\usepackage[titletoc]{appendix}
\usepackage{enumerate}

\fancypagestyle{plain}
{
	\fancyhf{}
	
}

\titleformat{\subsection}{\it}{\thesubsection.\enspace}{1.5pt}{}
\titleformat{\subsubsection}{\it}{\thesubsubsection.\enspace}{1.5pt}{}

\newtheorem{theorem}{Theorem}[section]

\newtheorem{lemma}[theorem]{Lemma}

\newtheorem{remark}[theorem]{Remark}

\numberwithin{equation}{section}

\newcommand\sn{{\bf{S}}({\bf{v}}^{n-1}_{1})}
\newcommand\an{{\bf{A}}({\bf{v}}^{n-1}_{1})}
\newcommand\bn{{\bf{B}}({\bf{v}}^{n-1}_{1})}
\newcommand\vn{{\bf{v}}^{n}}
\newcommand\vnn{{\bf{v}}^{n-1}}

\newcommand\fn{{\bf{F}}({\bf{v}}_{1}^{n-1}, \partial_{y}{\bf{v}}_{1}^{n-1})}

\DeclareMathOperator{\divv}{div}

\allowdisplaybreaks

\begin{document}
	\title{Local-in-time well-posedness for 2D   compressible  magneto-micropolar boundary layer in Sobolev spaces \hspace{-4mm}}
	\author{Yuming Qin$^{\ast}$$^{a,b}$ \quad Junchen Liu$^{b}$ \\[10pt]
		\small {$^a$ Institute for Nonlinear Sciences, Donghua University, 201620, Shanghai, P. R. China}\\
		\small {$^b$ School of Mathematics and Statistics,  Donghua University,
	201620, Shanghai, P. R. China}
		\\[5pt]
	}

	\footnotetext{Email: \it $^{\ast}$   yuming@dhu.edu.cn, yuming\_qin@hotmail.com}
	\date{}
	
	\maketitle

	\begin{abstract}
		{

			In this paper, we  study the two-dimensional  compressible magneto-micropolar boundary layer equations on the half-plane, which are  derived from 2D compressible magneto-micropolar fluid equations with the non-slip boundary condition on velocity,
			Dirichlet boundary condition on micro-rotational velocity and perfectly conducting boundary condition on magnetic field. 
				Based on a nonlinear coordinate transformation  proposed in \cite{LXY2019}, we first prove  the local-in-time well-posedness for the compressible magneto-micropolar boundary layer system in
			 Sobolev spaces, provided that initial tangential magnetic field is non-degenerate.
		}
		
		\vspace*{5pt}
		
	\vspace*{5pt}
	\end{abstract}

%
	
\section{Introduction}
The micropolar fluid theory, proposed by Turkish scholar  Eringen, A.-C. \cite{E1967} in the 1960s, extends classical fluid mechanics to describe fluids with microstructures or intrinsic rotational characteristics. While classical fluid mechanics, governed by the Navier-Stokes equations, assumes a continuum and neglects the microscopic rotations, particle interactions, and angular momentum inherent in certain fluids, this limitation \cite{HF1964,VP1964} makes it inadequate for capturing the behavior of complex fluids with microstructural or directional properties, such as liquid crystals, suspensions, and biological fluids. The micropolar fluid theory, especially the compressible case (see \cite{HV2009,PVH2012,R1997}),  addresses these deficiencies, providing a more comprehensive framework for studying such materials.

The behavior of a three-dimensional compressible micropolar fluid influenced by a magnetic field is stated as the following magneto-micropolar fluid equations \cite{WGL2017,JZJ2023}, which characterize the motion of the fluid as it interacts with the magnetic field,
\begin{equation}
	\left\{\begin{aligned}
	 &\partial_t\rho+\divv(\rho \mathbf{u})=0,\\
	&\partial_t(\rho \mathbf{u})+\divv(\rho \mathbf{u}\otimes \mathbf{u})-(\mu+\zeta)\Delta \mathbf{u}-(\mu+\lambda-\zeta)\nabla \divv \mathbf{u}+\nabla p
	=2\zeta \nabla\times \mathbf{w}
	+(\nabla \times \mathbf{H})\times \mathbf{H},\\
	&\partial_t(\rho \mathbf{w})+\divv(\rho \mathbf{u}\otimes \mathbf{w})
	-\mu'\Delta \mathbf{w}-(\mu'+\lambda')\nabla \divv \mathbf{w}
	+4\zeta \mathbf{w}=2\zeta \nabla \times \mathbf{u},\\
&	\partial_t\mathbf{H}-\nabla\times(\mathbf{u}\times \mathbf{H})=-\nabla\times(\sigma\nabla\times \mathbf{H}),\\
	&\divv \mathbf{H}=0.
	\end{aligned}\right.	
	\label{3d magneto-micropolar fluid}
\end{equation}
The unknowns $\rho$, $\mathbf{u}=(u_1, u_2, u_3)$, $\mathbf{w}=(w_x, w_y, w_{1})$, $p$ and $\mathbf{H}=(h_1, h_2, h_3)$ stand for the fluid density, velocity, micro-rotational velocity, pressure and magnetic field, respectively. 

As an intermediate case, we are concerned with  2D magneto-micropolar fluid equations. Then, the velocity  and magnetic field can be understood as $\mathbf{u} = (u_1(x, y, t), u_2(x,
y, t), 0)$ and $\mathbf{H} = (h_1(x, y, t), h_2(x, y, t), 0)$. We also assume that the rotational axis of particles is parallel to the $Z$-axis, namely, $\mathbf{w} = (0, 0, w_{1}(x, y, t))$. 
The equation \eqref{3d magneto-micropolar fluid}, after being transformed simply by $\eqref{3d magneto-micropolar fluid}_{1}$, is written as the following equation, defined in a periodic domain
${(t, x, y) \mid t \in [0, T], x \in \mathbb{T}, y \in \mathbb{R}^{+}}$:
\begin{equation}
	\left\{\begin{aligned}
		&\partial_t\rho+\divv(\rho \mathbf{u})=0,\\
		&\rho\big(\partial_t\mathbf{u}+(\mathbf{u}\cdot\nabla)\mathbf{u}\big)-(\mu+\zeta)\Delta \mathbf{u}-(\mu+\lambda-\zeta)\nabla \divv \mathbf{u}+\nabla p
		=2\zeta \nabla^{\bot}{w_{1}}
		+(\nabla \times \mathbf{H})\times \mathbf{H},\\
		&\rho\big(\partial_tw_{1}+(\mathbf{u}\cdot\nabla)w_{1}\big)
		-\mu'\Delta {w_{1}}-(\mu'+\lambda')\nabla \divv {w_{1}}
		+4\zeta {w_{1}}=2\zeta \nabla \times \mathbf{u},\\
		&	\partial_t\mathbf{H}-\nabla\times(\mathbf{u}\times \mathbf{H})=-\nabla\times(\sigma\nabla\times \mathbf{H}),\\
		&\divv \mathbf{H}=0,
	\end{aligned}\right.	
	\label{2d magneto-micropolar fluid}
\end{equation}
where, the original three-dimensional unknowns are redefined,  then $\rho=\rho(x,t)$, $\mathbf{u} = (u_1(x, y, t), u_2(x,
y, t))$, $\mathbf{H} = (h_1(x, y, t), h_2(x, y, t))$ and $\nabla^{\bot} w_{1}=(-\partial_{y}w_{1}, -\partial_{x}w_{1})$. 
The function $p$ is $\gamma$-law pressure \cite{XTWT2022}, which takes the  following form
\begin{align}
	p=a\rho^{\gamma}, ~~ a>0, ~~ \gamma>1.
	\label{p}
\end{align}

 Furthermore, the constants $\mu$ and $\lambda$ stand for the shear and the bulk viscosity
coefficients of the flow; the constant $\zeta$ denotes the dynamics 
microrotation viscosity; $\mu'$ and $\lambda'$ are the angular
viscosities; the resistivity coefficient $\sigma$ is inversely proportional to the electrical conductivity constant. These parameters satisfy
$$
\mu, \zeta, \mu',  \sigma >0, \qquad 2\mu+3\lambda-4\zeta \geq 0, \qquad 2\mu'+3\lambda' \geq 0.
$$

This system \eqref{2d magneto-micropolar fluid} serves as a model for describing the motion of aggregates of small solid ferromagnetic particles relative  governed by the compressible magneto-micropolar  fluid equations, where the fluids have nonconstant density.      Consequently, it is considerably more complex than the classical incompressible magneto-micropolar fluid equations with constant density.      To fully define the system \eqref{2d magneto-micropolar fluid}, the following boundary conditions are imposed:
\begin{align}
\mathbf{u}|_{y=0}=\mathbf{0}, \quad w_{1}|_{y=0}=0, \quad (\partial_{y}h_{1}, h_{2})|_{y=0}=\mathbf{0}.
\label{boundary condition 1}
\end{align}
Here   the velocity, micro-rotational velocity  and magnetic field
satisfy the no-slip boundary condition, the Dirichlet boundary condition and  the perfectly conducting wall boundary condition, 
 respectively.

Motivated by the Prandtl boundary layer theory, this paper aims to examine the asymptotic behavior of the flow near the physical boundary $\{y=0\}$ for the equation \eqref{2d magneto-micropolar fluid} as 
the shear viscosity $\mu$,  the bulk viscosity  $\lambda$, 
the dynamics  microrotation viscosity $\zeta$,   the angular
viscosities $\mu', \lambda'$ and the resistivity coefficient $\sigma$
tend to zero.  In a rigorous sense, as the parameters vanish, equation \eqref{2d magneto-micropolar fluid} can be transformed into the 2D compressible ideal magneto-micropolar system:
\begin{equation}
	\left\{
	\begin{aligned}
		&\partial_t\rho^e +\divv(\rho^e \mathbf{u}^e)=0,\\
		&\rho^e\big(\partial_t\mathbf{u}^e+(\mathbf{u}^e\cdot\nabla)\mathbf{u}^e\big)+\nabla p(\rho^e)	-(\nabla\times \mathbf{H}^e)\times \mathbf{H}^e=0,\\
		&\rho^e\big(\partial_tw^e_1+(\mathbf{u}^e\cdot\nabla)w^e_{1}\big)=0,\\
		& \mathbf{H}^e_t-\nabla\times (\mathbf{u}^e\times \mathbf{H}^e)=0,\quad \divv \mathbf{H}^e=0,
		\\
		&(u_{2}, h_{2})|_{y=0}=\mathbf{0}.
	\end{aligned}
	\right.
\label{ideal}
\end{equation}
The inconsistency of boundary conditions between equations \eqref{boundary condition 1} and $\eqref{ideal}_{5}$ on the boundary $\{y=0\}$ leads to the formation of a thin boundary layer in the vanishing limits of the parameters, where the tangential velocity, micro-rotational velocity, and tangential magnetic field exhibit dramatic changes.

Prior to presenting a detailed analysis of boundary layer behavior,
we first review some known results from the established theoretical framework. The  boundary layer theory was first initiated by Prandtl,  L.  in 1904. 
Specifically, he conducted a formal analysis and derived a coupled elliptic-parabolic system, now known as the Prandtl equations, to describe the boundary layer phenomenon.
Early studies of the Prandtl equations primarily focused on two-dimensional incompressible flows, with considerable progress firstly made by Oleinik, O.-A. \cite{Oleinik-1966}. 
She proved the local-in-time well-posedness in H\"older spaces for  2D Prandtl equations under a monotonicity condition imposed on the tangential velocity. 
These results, along with an expanded introduction to boundary layer theory, were later presented in the classic monograph \cite{Oleinik-Samokhin-1999} by Oleinik and her collaborators.
By using  a so-called  Crocco transformation developed in \cite{Oleinik-1966,Oleinik-Samokhin-1999},  Xin and Zhang \cite{Xin-Zhang-2004} obtained a global existence of BV weak solutions to  2D unsteady Prandtl system  with the addition of favorable condition $(\partial_{x}P \leq 0)$ on pressure. 
Motivated by a direct energy method, instead of considering Crocco transformation, which can  recover Oleinik's well-posedness results,  Alexandre, Wang, Xu and  Yang \cite{Alexandre-Wang-Xu-Yang-2015}
proved that the solution exists locally with respect to time in the weight Sobolev spaces by  applying Nash-Moser iteration.
Masmoudi and Wong \cite{Masmoudi-Wong-2015}, distinct from \cite{Alexandre-Wang-Xu-Yang-2015}, established the local existence of solutions by leveraging a key cancellation property in the convection terms that overcomes the loss of $x$-derivatives in the tangential direction under the monotonicity assumption.
Still, there
are some results on the Prandtl equations  under the monotonicity assumption, see \cite{Xu-Zhang-2017,QL2025,Fan-Ruan-Yang-2021,Wang-Xie-Yang-2015,Q2025}.
In violation of Oleinik's monotonicity setting, blow-up of solutions, some instability and ill-posedness mechanisms are unfiltered out, cf. \cite{Grenier-2000,Hong-Hunter-2003,Gerard-Varet-Dormy-2010,Liu-Yang-2017,Gerard-Varet-Nguyen-2010,Guo-Nguyen-2011,E-Engquist-1997}.
Without Oleinik's monotonicity assumption on   2D cases, the solutions  and  initial data are
desired to be in the analytic or Gevrey class, cf. \cite{Sammartino-Caflisch-1-1998,Lombardo-Cannone-Sammartino-2003,Igntova-Vicol-2016,Paicu-Zhang-2021} for the framework of the analyticity and \cite{Gerard-Varet-Masmoudi-Gevrey-2015,Li-Yang-2020,Wang-Wang-Zhang-2024,Dietert-Gerard-Varet-2019} for the framework of the Gevrey class.

When $w_{1}=0$ and $\rho=\rm{constant}$, then the equations \eqref{2d magneto-micropolar fluid} reduces to 2D incompressible magnetohydrodynamic (MHD) equations.
Considering the high Reynolds number limit (the magnetic Reynolds number and the fluid Reynolds number tend to infinity at the same rate) of incompressible viscous MHD equations, then 2D incompressible MHD boundary layer equations can be derived from the non-slip boundary condition to velocity and perfectly conducting boundary condition to magnetic field. Recently, there are many mathematical results on two dimensional MHD boundary layer
system,
the local-in-time well-posedness was
obtained in Sobolev space without any monotonicity assumption on the tangential velocity under the condition of non-degeneracy initial horizontal magnetic field by \cite{LWXY2019,LXY2019}.
There are further works concerned with the global well-posedness of  2D MHD  boundary layer equations in the Sobolev spaces (cf.  \cite{CRWZ2020}),  analytic functions (cf. \cite{LZ2021}) and  Gevrey class (cf. \cite{TW2023}).

Many important problems on the magneto-micropolar boundary layer system remain open. 
For the incompressible  case,
Lin and Zhang \cite{LZ2021mm} proved  the local well-posedness to  2D magneto-micropolar boundary layer equations  with the Dirichlet boundary condition on  magnetic fields in the framework of analytic functions. 
If the  magnetic field is endowed with the perfectly conducting wall boundary condition and the initial tangential magnetic field is nondegenerate, Lin, Liu and Zhang \cite{LLZ2024} established the
local-in-time well-posedness  for the magneto-micropolar boundary layer equations in Sobolev spaces without monotonicity with the lower regularity initial data in $H^3$.
In the  Gevrey function spaces without any structural assumption, the local well-posedness of the two-dimensional magneto-micropolar
boundary layer system was established based on a new cancellation mechanism by Tan and Zhang \cite{TZ2022}.
On the other hand, for 2D magneto-micropolar boundary layer equations  without resistivity, the local well-posedness
theory in Sobolev spaces is established by Zou, Lin \cite{ZL2022} and
Zhong, Zhang \cite{ZZ2025}.

Compared to incompressible models, compressible boundary layer equations present greater analytical challenges due to their strong nonlinearities and complex interactions between physical quantities.
To date, few well-posedness results have been established for the compressible versions of the Prandtl equations and MHD boundary layer equations. The limited available results can be categorized into:

$\bullet$ compressible Prandtl equations
\begin{itemize}
\item[(1)] 
 Wang,  Xie,   Yang (2015)	\cite{Wang-Xie-Yang-2015} and 
 Fan, Ruan,  Yang (2021) 	\cite{Fan-Ruan-Yang-2021}: 	
 local well-posedness of 2D compressible Prandtl boundary layer
equations in Sobolev space;
\item[(2)] Ding and Gong (2017) \cite{DG2017}: 
global existence of weak solution to 2D compressible Prandtl equations;

\item [(3)] 
Chen, Huang and Li (2024) \cite{CHL2024}:
 global existence  and the large-time decay estimates of strong solution to 2D compressible Prandtl equations with small
 initial data, which is analytical in the tangential variable;
 
 \item [(4)]
Chen,  Ruan and   Yang (2025) \cite{CRY2025}:
 local well-posedness of  solutions to 3D compressible Prandtl equations, which is real-analytic in tangential direction and Sobolev regular in normal direction.
\end{itemize}

$\bullet$ compressible MHD boundary layer equations:
\begin{itemize}
	
	\item[(5)]  Huang, Liu and Yang (2019) \cite{HLY2019}:  local-in-time well-posedness of 2D MHD boundary layer system in
	weighted Sobolev spaces  under the non-degeneracy condition on the tangential magnetic field;
	
	\item[(6)] Li and Xie (2024) \cite{LX2024}: long time well-posedness of  2D compressible MHD boundary layer equations in Sobolev space. 
\end{itemize}
To our best knowledge, so far there is no result  concerning on the 
well-posedness of solutions for  2D compressible magneto-micropolar boundary layer equations.  This gap in theoretical understanding, combined with existing results from Prandtl and MHD boundary layer analyses, motivates our initial exploration of the compressible magneto-micropolar system.

In this paper, the main purpose
 is to achieve that  the local-in-time well-posedness for 2D compressible magneto-micropolar boundary layer equations.
To identify the terms in \eqref{ideal} whose contribution is pivotal for the boundary layer, we use the following scalings inspired by \cite{Oleinik-Samokhin-1999},
\begin{align}
\mu+\zeta=\varepsilon\tilde{\mu}, \quad \zeta={\sqrt{\varepsilon}}\tilde{\zeta}, \quad	(\lambda,  \mu',  \lambda', \sigma) = \varepsilon ( \tilde{\lambda}, \tilde{\mu'}, \tilde{\lambda'}, \tilde{\sigma}),
\end{align}
for the positive constants $\tilde{\mu},  \tilde{\lambda}, \tilde{\zeta}, \tilde{\mu'}, \tilde{\lambda'}, \tilde{\sigma}$.
Based on Prandtl's assertions, the thickness of the boundary layer is of order \( \sqrt{\varepsilon} \), and its behavior is governed by the Prandtl-type equations, which can be derived from the compressible magneto-micropolar system \eqref{2d magneto-micropolar fluid}-\eqref{boundary condition 1}. More precisely, inside the boundary layer, we introduce the rescalings
\[
\tilde{t} = t, \quad \tilde{x} = x, \quad \tilde{y} = \frac{y}{\sqrt{\varepsilon}},
\]
and define the new unknown functions as follows:
\[
(\tilde{\rho}, \tilde{u}_1, \tilde{w}_{1}, \tilde{h}_1)(\tilde{t}, \tilde{x}, \tilde{y}) = (\rho, u_1, w_{1}, h_1)(t, x, y), \quad (\tilde{u}_2, \tilde{h}_2)(\tilde{t}, \tilde{x}, \tilde{y}) = \frac{1}{\sqrt{\varepsilon}}(u_2, h_2)(t, x, y).
\]
Substituting this ansatz into the equations \eqref{2d magneto-micropolar fluid}, the leading-order terms yield
\begin{equation}
	\left\{
	\begin{aligned}
		&\partial_{t}\rho+(u_1\partial_{x}+u_2\partial_{y})\rho+\rho(\partial_{x} u_1+\partial_{y} u_2)=0,\\
		&\rho\{\partial_{t}u_1+(u_1\partial_{x}+u_2\partial_{y})u_1\}+\partial_{x}(p+\frac{1}{2}h^2_1)-(h_1\partial_{x}+h_2\partial_{y}) h_1-\mu\partial_{y}^2u_1=0,\\
		&\partial_{y}(p+\frac{1}{2}h^2_1)=0,\\
		&\rho\{\partial_{t} w_{1}+(u_1\partial_{x}+u_2\partial_{y})w_{1}\}+2\zeta\partial_{y}u_{1}-\mu'\partial_{y}^2w_{1}
		=0,\\
		&\partial_{t}h_1+\partial_{y}(u_2h_1-u_1h_2)-\sigma\partial_{y}^2h_1=0,\\
		&\partial_{t}h_2-\partial_{x}(u_2h_1-u_1h_2)-\sigma\partial_{y}^2h_2=0,\\
		&\partial_{x}h_1+\partial_y h_2=0,
		\\
		&(u_{1}, u_{2}, w_{1}, \partial_{y}h_{1}, h_{2})|_{y=0}=\mathbf{0},
		\\
		&\lim_{{y}\rightarrow +\infty}(\rho, u_1, w_{1}, h_1)({t},{x},{y})=(\rho^{e}, u^{e}_{1}, w^{e}_{1}, q_{1}^{e})({t},{x}, 0).
	\end{aligned}
	\right.
\label{mm boundary layer}
\end{equation}
From $\eqref{mm boundary layer}_{3}$ and $\eqref{mm boundary layer}_{9}$, we have
\begin{align}
\left(p+\frac{1}{2}h^2_1\right)(t, x, y) \equiv \left(a(\rho^{e})^\gamma+\frac {1}{2}(h_{1}^{e})^2\right)(t, x, 0)=P(t, x),
\label{1.1}
\end{align}
which means that
\begin{align}
	p(t, x, y)=P(t,x)-\frac{1}{2}h_1^2(t,x,y)>0.
	\label{1.2}
\end{align}
Moreover, it follows by virtue of \eqref{p} that
\begin{align}
	\frac{\partial_i{\rho}}{\rho}=\frac{\partial_ip}{\gamma p},\quad i=t, x, y.
\label{1.3}
\end{align}
 Plugging  \eqref{1.2}-\eqref{1.3} into $\eqref{mm boundary layer}_{1}$, we arrive at
\begin{align}
\begin{split}
	\partial_{x}u_{1}+\partial_{y}u_{2}
	&=-\frac{\partial_{t} p+(u_{1}\partial_{x}+u_{2}\partial_{y})p}{\gamma p}\\
	&=-\frac{P_{t}-h_{1}\partial_{t}h_{1}+P_{x}u_{1}-h_{1}(u_{1}\partial_{x}+u_{2}\partial_{y})h_{1}}{\gamma(P-\frac{1}{2} h_{1}^2)}
	.
\end{split}
\label{1.4}
\end{align}
By   $\eqref{mm boundary layer}_{5}$-$\eqref{mm boundary layer}_{8}$, it is straightforward to check that 
\begin{align}
\partial_th_1+(u_1\partial_x+u_2\partial_y)h_1
=-h_1(\partial_x u_1+\partial_y u_2)+(h_1\partial_x+h_2\partial_y)u_1+\sigma\partial_{y}^2h_1,
\label{1.5}
\end{align}
which, together with \eqref{1.1}, implies
\begin{align}
		\left(\frac{h_1^{2}}{\gamma(P-\frac{1}{2} h_{1}^2)}+1\right) \partial_{x}u_{1}+\partial_{y}u_{2}
		=-\frac{P_{t}+P_{x}u_{1}}{\gamma(P-\frac{1}{2} h_{1}^2)}+
		\frac{h_{1}(h_1\partial_x+h_2\partial_y)u_1+\sigma h_{1}\partial_{y}^2h_1}{\gamma(P-\frac{1}{2} h_{1}^2)},
	\label{1.6}
\end{align}
and then, we have
\begin{align}
	 \partial_{x}u_{1}+\partial_{y}u_{2}
	=-\frac{P_{t}+P_{x}u_{1}}{\gamma(P-\frac{1}{2} h_{1}^2)+h_{1}^2}+
	\frac{h_{1}(h_1\partial_x+h_2\partial_y)u_1+\sigma h_{1}\partial_{y}^2h_1}{\gamma(P-\frac{1}{2} h_{1}^2)+h_{1}^2}.
	\label{1.7}
\end{align}
Thus, boundary value problem \eqref{mm boundary layer} after imposing  the initial data can be rewritten in the following form
\begin{equation}
	\left\{
	\begin{aligned}
		&\partial_{t}u_1+(u_1\partial_{x}+u_2\partial_{y})u_1+\left( \frac{a}{P-\frac{1}{2}h^2_1}\right)^{\frac{1}{\gamma}} P_{x}-\left( \frac{a}{P-\frac{1}{2}h^2_1}\right)^{\frac{1}{\gamma}} (h_1\partial_{x}+h_2\partial_{y}) h_1\\
		&\quad-\mu\left( \frac{a}{P-\frac{1}{2}h^2_1}\right)^{\frac{1}{\gamma}} \partial_{y}^2u_1=0,\\
		&\partial_{t} w_{1}+(u_1\partial_{x}+u_2\partial_{y})w_{1}+2\zeta\left( \frac{a}{P-\frac{1}{2}h^2_1}\right)^{\frac{1}{\gamma}} \partial_{y}u_{1}-\mu'\left( \frac{a}{P-\frac{1}{2}h^2_1}\right)^{\frac{1}{\gamma}} \partial_{y}^2w_{1}
		=0,\\
		&\partial_th_1+(u_1\partial_x+u_2\partial_y)h_1
		-h_1\frac{P_{t}+P_{x}u_{1}}{\gamma(P-\frac{1}{2} h_{1}^2)+h_{1}^{2}}
		-\frac{ \gamma(P-\frac{1}{2} h_{1}^2)}{ \gamma(P-\frac{1}{2} h_{1}^2)+h_{1}^2}(h_1\partial_x+h_2\partial_y)u_1
		\\
		&\quad-\sigma\frac{ \gamma(P-\frac{1}{2} h_{1}^2)}{ \gamma(P-\frac{1}{2} h_{1}^2)+h_{1}^2}\partial_{y}^2h_1=0,\\
		&\partial_{x}u_{1}+\partial_{y}u_{2}
		=-\frac{P_{t}+P_{x}u_{1}}{\gamma(P-\frac{1}{2} h_{1}^2)+h_{1}^2}+
		\frac{h_{1}(h_1\partial_x+h_2\partial_y)u_1+\sigma h_{1}\partial_{y}^2h_1}{ \gamma(P-\frac{1}{2} h_{1}^2)+h_{1}^2},\\
		&\partial_{x}h_1+\partial_y h_2=0,\\
		&(u_{1}, w_{1}, h_{1})|_{t=0}=(u_{1,0}, w_{1,0}, h_{1,0})(x,y),\\
		&(u_{1}, u_{2}, w_{1}, \partial_{y}h_{1}, h_{2})|_{y=0}=\mathbf{0},
		\\
		&\lim_{{y}\rightarrow +\infty}( u_1, w_{1}, h_1)({t},{x},{y})=(\rho^{e}, u^{e}_{1}, w^{e}_{1}, q_{1}^{e})({t},{x}, 0)=(U, I, H)(t,x),
	\end{aligned}
	\right.
	\label{1 mm boundary layer}
\end{equation}
where the outflow  $U(x,t), I(x,t)$ and $H(x,t)$ are the trace of the tangential velocity, which satisfy the Bernoulli's law
\begin{equation}
	\left\{
	\begin{aligned}
		&U_t +UU_x +\left( \frac{a}{P-\frac{1}{2}H^{2}}\right)^{\frac{1}{\gamma}}P_{x}-\left( \frac{a}{P-\frac{1}{2}H^{2}}\right)^{\frac{1}{\gamma}}HH_x=0,\\
		&I_t+UI_x=0,\\
		&H_t+UH_x
		-\frac{P_{t}+P_{x}U}{\gamma(P-\frac{1}{2} H^2)+H^2}H
		-\frac{\gamma(P-\frac{1}{2} H^2)}{\gamma(P-\frac{1}{2} H^2)+H^2}HU_x
		=0.
	\end{aligned}
\label{berlaw}
	\right.
\end{equation}

The main result of this paper is stated in the following theorem.
\begin{theorem}\label{t1}
Suppose the  outflow $(U, I, H, P)(t, x)$ in the equations \eqref{berlaw} is smooth, and  the initial data $(u_{1,0},w_{1,0}, h_{1,0})(x,y)$  are smooth, compatible with the boundary condition and  satisfy
\begin{align}
 h_{1,0}(x,y)\geq 2\delta, \quad \frac{1}{2}\big(h_{1,0}(x,y)\big)^2 \leq P(0,x)-2\delta,
\label{im}
\end{align}
for $t\in[0,T],  (x,y)\in\mathbb{T}\times\mathbb{R}_+$ with some constant $\delta>0$.
Then there exists  $0< T_{*} \leq T$ such that the initial-boundary value problem \eqref{1 mm boundary layer} admits  a unique classical solution $(u_{1}, u_{2}, w_{1}, h_{1}, h_{2} )$  
satisfying 
\begin{align*}
 h_{1}(t,x,y)\geq \delta, ~  ~  \frac{1}{2}h_{1}^2(x,y) \leq P(t,x)-\delta,
\end{align*}
in the region $\displaystyle D_{T_*}:=\{(t,x,y)|t\in[0,T_*], x \in \mathbb{T}, y \in \mathbb{R}_+\}$.
Furthermore, we have that
$(u_{1}, w_{1}, h_{1}), \partial_y(u_{1}, w_{1}, h_{1})$
and $\partial_y^2(u_{1}, w_{1}, h_{1})$ are continuous and bounded in $D_{T_*}$. Additionally,
$\partial_t (u_{1}, w_{1}, h_{1})$, $\partial_x(u_{1}, h_{1})$, $(u_{2}, h_{2})$ and $\partial_y(u_{2}, h_{2})$ are continuous and bounded in any compact set of $D_{T_*}$.
\end{theorem}

 The remainder of the paper is organized as follows.
 In Sections \ref{s2},   we reformulate the initial-boundary value problem \eqref{1 mm boundary layer} to overcome the difficulty of  loss of $x$-derivatives by a coordinate transformation, along with
 stating  some elementary lemmas. In Section \ref{s3}, we derive a priori estimate and  obtain the well-posedness theory on the reformulated compressible magneto-micropolar boundary layer equations.
Finally, local-in-time existence and uniqueness of the solution for the original boundary layer equations \eqref{1 mm boundary layer} will be studied
in Section \ref{s4}.
\section{Preliminary and elementary lemma}\label{s2}

\subsection{Reformulation of the system }
The main difficulty in the local well-posedness theory for the initial-boundary value problem \eqref{1 mm boundary layer} stems from the derivative loss in the $x$-direction for the normal velocity field  $u_{2}$ and the normal magnetic field $h_{2}$.
To address this issue, inspired by \cite{HLY2019,LXY2019}, we utilize the incompressible condition associated with the magnetic field in \eqref{1 mm boundary layer} and introduce a stream function, denoted by $\psi(t,x,y)$, which incorporates both the tangential and normal components of the magnetic field. Specifically, the stream function satisfies the following relations: 
\begin{align}
	h_1=\partial_y\psi,\qquad h_2=-\partial_x\psi,\qquad~ \psi|_{y=0}=0.
\label{stream-function}
\end{align}
Moreover, combining the  equation \eqref{1.5} with the boundary condition $\eqref{1 mm boundary layer}_{7}$, one may directly verify  that $\psi(t,x,y)$ is governed by the following equation obeys the evolution equation:
\begin{align}
\partial_t\psi+(u_1\partial_x+u_2\partial_y)\psi=\sigma\partial_{y}^2\psi.
\label{stream-equation}
\end{align}
Under the assumption of non-zero tangential magnetic component:
\begin{align}
	h_1=\partial_y\psi>0,
\label{2.1}
\end{align}
we can introduce the following invertible transformation
\begin{align}
	\bar{t}=t, \qquad  \bar{x}=x, \qquad \bar{y}=\psi(t,x,y),
\label{2.2}
\end{align}
and new unknown functions
\begin{align}
	(\hat{u}_1, \hat{w}_{1}, \hat{h}_1)(\bar{t},\bar{x},\bar{y}):=(u_1,w_{1},h_1)(t,x,y).
	\label{2.3}
\end{align}

Under this new coordinate, the region $\{(t, x, y)|t\in [0, T], x\in \mathbb{T},  y\in\mathbb{R}_+\}$ is mapped into $\{(\bar{t}, \bar{x}, \bar{y})|\bar{t}\in [0, T],  \bar{x}\in\mathbb{T},  \bar{y}\in\mathbb{R}_+\}$, and the boundary of $\{ y=0\}$ ($\{ y=+\infty\}$ respectively) becomes the boundary of $\{ \bar{y}=0\}$ ($\{ \bar{y}=+\infty\}$ respectively). Also initial-boundary value problem \eqref{1 mm boundary layer} for $(\hat{u}_1, \hat{w}_{1},\hat{h}_1)(\bar{t},\bar{x},\bar{y})$  can be transformed into
\begin{equation}
	\left\{
	\begin{aligned}
		&\partial_{\bar{t}}\hat{u}_1+\hat{u}_1\partial_{\bar{x}}\hat{u}_1+\left( \frac{a}{P-\frac{1}{2}\hat{h}^2_1}\right)^{\frac{1}{\gamma}} P_{\bar{x}}-\left( \frac{a}{P-\frac{1}{2}\hat{h}^2_1}\right)^{\frac{1}{\gamma}} h_1\partial_{\bar{x}} h_1+\left\lbrace \sigma-\mu\left( \frac{a}{P-\frac{1}{2}\hat{h}^2_1}\right)^{\frac{1}{\gamma}}\right\rbrace  \hat{h}_{1}\partial_{\bar{y}}\hat{h}_{1} \partial_{\bar{y}}\hat{u}_{1} \\
		&\quad-\mu\left( \frac{a}{P-\frac{1}{2}\hat{h}^2_1}\right)^{\frac{1}{\gamma}}
		\hat{h}_{1}^{2} \partial_{\bar{y}}^2\hat{u}_1=0,\\
		&\partial_{\bar{t}}\hat{w}_1+\hat{u}_1\partial_{\bar{x}}\hat{w}_1+2\zeta\left( \frac{a}{P-\frac{1}{2}\hat{h}^2_1}\right)^{\frac{1}{\gamma}}\hat{h}_{1} \partial_{\bar{y}}\hat{u}_{1}
		+\left\lbrace \sigma-\mu'\left( \frac{a}{P-\frac{1}{2}\hat{h}^2_1}\right)^{\frac{1}{\gamma}}\right\rbrace  \hat{h}_{1}\partial_{\bar{y}}\hat{h}_{1} \partial_{\bar{y}}\hat{w}_{1}\\
		&\quad-\mu'\left( \frac{a}{P-\frac{1}{2}\hat{h}^2_1}\right)^{\frac{1}{\gamma}}
		\hat{h}_{1}^{2} \partial_{\bar{y}}^2\hat{w}_{1}
		=0,\\
		&\partial_{\bar{t}}\hat{h}_{1}+\hat{u}_1\partial_{\bar{x}}\hat{h}_{1}
		-\hat{h}_{1}\frac{P_{t}+P_{x}\hat{u}_{1}}{\gamma(P-\frac{1}{2} \hat{h}_{1}^2)+\hat{h}_{1}^{2}}
		-\frac{ \gamma(P-\frac{1}{2} \hat{h}_{1}^2)}{ \gamma(P-\frac{1}{2} \hat{h}_{1}^2)+\hat{h}_{1}^2}\hat{h}_{1}\partial_{\bar{x}}\hat{u}_1
		+
		\sigma\left\lbrace 1-\frac{ \gamma(P-\frac{1}{2} \hat{h}_{1}^2)}{ \gamma(P-\frac{1}{2} \hat{h}_{1}^2)+\hat{h}_{1}^2}\right\rbrace  \hat{h}_{1}(\partial_{\bar{y}}\hat{h}_{1})^{2}\\
		&\quad
		-\sigma\frac{ \gamma(P-\frac{1}{2} \hat{h}_{1}^2)}{ \gamma(P-\frac{1}{2} \hat{h}_{1}^2)+\hat{h}_{1}^2}\hat{h}_{1}^{2}\partial_{y}^2\hat{h}_{1}=0,
		\\
		&(\hat{u}_1, \hat{w}_1, \partial_{\bar{y}}\hat{h}_{1})|_{\bar{y}=0}=\mathbf{0}, \\ &\lim_{\bar{y}\rightarrow +\infty}(\hat{u}_1, \hat{w}_{1}, \hat{h}_{1})(\bar{t},\bar{x},\bar{y})=(U,I,H)(\bar{t},\bar{x}),\\
		&(\hat{u}_{1},\hat{w}_{1},\hat{h}_{1})|_{\bar{t}=0}=(\hat{u}_{1,0},\hat{w}_{1,0}, \hat{h}_{1,0})(\bar{x},\bar{y}).
	\end{aligned}
	\label{hat}
	\right.
\end{equation}

\subsection{Some notations and elementary lemmas}
In this subsection, we first begin by introducing the Sobolev spaces required for our next analysis.
Set
\begin{align*}
	\Omega=\{(x,y): x\in\mathbb{T}, y\in\mathbb{R}_+\},
\end{align*}
and
\begin{align*}
	\Omega_T=[0,T]\times\Omega=\{(t,x,y): t\in[0,T], x\in\mathbb{T}, y\in\mathbb{R}_+\}.
\end{align*}
Denote by $H^k(\Omega)$ the classical Sobolev spaces of  function $f\in H^k(\Omega)$ such that 
$$
\|f\|_{H^k(\Omega)}:=\left(\sum_{|\alpha_1|+|\alpha_2|\leq k}
\|\partial^{\alpha_1}_x\partial^{\alpha_2}_y f\|^2_{L^2(\Omega)}\right)^\frac{1}{2}<\infty.
$$
The derivative operator with multi-index $\alpha = (\alpha_1, \alpha_2, \alpha_3) \in \mathbb{N}^3$ is denoted as
\begin{align*}
	\partial^\alpha=\partial_t^{\alpha_1}\partial_x^{\alpha_2}\partial_y^{\alpha_3},\quad\mbox{with}\quad |\alpha|=\alpha_1+\alpha_2+\alpha_3.
\end{align*}
The Sobolev space and norm are thus defined as
\begin{align*}
	\mathcal{H}^k(\Omega_T)=\left\lbrace f(t,x,y): \Omega_T\rightarrow \mathbb{R}, \quad \|f\|_{\mathcal{H}^k(\Omega_T)}<\infty\right\rbrace 
\end{align*}
with
\begin{align*}
	\|f(t)\|_{\mathcal{H}^k(\Omega)}=\sup_{0\leq t< T}\|f(t)\|_{\mathcal{H}^k(\Omega)}:=\sup_{0\leq t< T}\left(\sum_{|\alpha|\leq k}
	\|\partial^{\alpha}f(t)\|^2_{L^2(\Omega)}\right)^\frac{1}{2}.
\end{align*}

Next, we give some useful inequalities whose proof  can be found in \cite{HLY2019,Masmoudi-Wong-2015,LXY2019}.

\begin{lemma}(Sobolev type inequality)\label{sobolevtype}
	Let $f:\Omega \rightarrow \mathbb{R}$.  Then there exists a constant $C>0$ such that
	\begin{eqnarray}
		\|f\|_{L^{\infty}(\Omega)}    \leq C
		\left(\|f\|_{L^{2}(\Omega )}   +\|\partial_{x}f\|_{L^{2}(\Omega)}  +\|\partial_{y }^{2}f\|_{L^{2}(\Omega )} \right).
	\end{eqnarray}
\end{lemma}

\begin{lemma}\label{trace}
For any proper  $f$ and $g$ in $\Omega_T$, if $\displaystyle \lim\limits_{y \rightarrow+\infty}(fg)(t, x, y)=0$, then the following  inequality holds
\begin{align}
	\left|\int_{\mathbb{T}}(fg)(t, x,0) dx \right|\leq\|\partial_y f(t,\cdot)\|_{L^2(\Omega)}\|g(t,\cdot)\|_{L^2(\Omega)}+\|\partial_y g(t,\cdot)\|_{L^2(\Omega)}\|f(t,\cdot)\|_{L^2(\Omega)},
	\label{trace1}
\end{align}
for all $t \in[0,T]$.
\end{lemma}

\begin{lemma}\label{inq}
	For suitable functions $f, g \in \mathcal{H}^k(\Omega_T)$ with some integer $k\geq 4$, it holds that for $|\alpha|+|\beta|\leq k$ and $t \in [0, T],$
	\begin{equation}
		\big\|(\partial^\alpha f\partial^\beta g)(t,\cdot) \big\|_{L^2(\Omega)}\leq C\|f(t)\|_{\mathcal{H}^k(\Omega)}\cdot\|g(t)\|_{\mathcal{H}^k(\Omega)}.\label{inq1}
	\end{equation}
	Then, it implies that for $|\alpha|\leq k$,
	\begin{align}
		\big\|\partial^\alpha(fg)(t,\cdot)\big\|_{L^2(\Omega)}
		\leq C \|f(t)\|_{\mathcal{H}^k(\Omega)}\cdot\|g(t)\|_{\mathcal{H}^k(\Omega)},
		\label{inq2}
	\end{align}
	and 
	\begin{align}
		&\big\|\big([\partial^\alpha, f]~g\big)(t,\cdot)\big\|_{L^2(\Omega)} \leq C \|f(t)\|_{\mathcal{H}^k(\Omega)}\cdot\|g(t)\|_{\mathcal{H}^{k-1}(\Omega)}.
		\label{inq3}
	\end{align}
\end{lemma}

Finally, the notation $\left\langle a, b \right\rangle \triangleq \int_{\Omega} a(x, y)b(x,y) dx dy$ means the $L^2$ inner product of $a, b$ on $\Omega$ throughout this paper.  In addition, we also use the notation $[A, B]$ to denote the commutator between $A$ and $B$.

\section{Well-posedness of solutions}\label{s3}

In this section, we are devoted to  establishing the well-posedness of solutions $(\hat{ u}_{1}, \hat{w}_{1}, \hat{q}_{1})$ for  the initial-boundary value problem \eqref{hat}. For notational simplicity, we  replace the notation of $(\bar{t},\bar{x},\bar{y})$ by $(t, x, y)$ and    omit the ``hat" symbols for the system \eqref{hat} in this section. In addition,
to facilitate the analysis of the system \eqref{hat}, it is advantageous to replace the original variable $h_{1}$ with a new function to simplify subsequent derivations. Specifically, we define
\begin{align}
	{q}_{1}= {q}_{1}({t},{x},{y}):=\frac{1}{2}{h}_1^2({t},{x},{y}),
	\label{qh}
\end{align}
and introduce the following notations
\begin{align*}
	\mathcal{A}=\left( \frac{a}{P-\frac{1}{2}{h}^2_1}\right)^{\frac{1}{\gamma}}=\left( \frac{a}{P-q_{1}}\right)^{\frac{1}{\gamma}}, \qquad Q={ \gamma(P-\frac{1}{2} {h}_{1}^2)+{h}_{1}^2}
	={ \gamma(P-q_{1})+2q_{1}},
\end{align*}
then the system \eqref{hat} can be formulated as the following problem:
\begin{equation}
	\left\{
	\begin{aligned}
		&\partial_{ {t}}u_1+u_1\partial_{ {x}}u_1+\mathcal{A} P_{ {x}}-\mathcal{A} \partial_{ {x}} q_{1}+\left( \sigma-\mu\mathcal{A}\right) \partial_{ {y}} q_{1} \partial_{ {y}}u_{1} -2\mu\mathcal{A}
		q_{1} \partial_{ {y}}^2u_1=0,\\
		&\partial_{ {t}}w_1+u_1\partial_{ {x}}w_1+2\zeta\mathcal{A} \sqrt{2 q_{1}} \partial_{ {y}}u_{1}
		+\left( \sigma-\mu'\mathcal{A}\right)  \partial_{ {y}} q_{1} \partial_{ {y}}w_{1}-2\mu'\mathcal{A}
		q_{1} \partial_{ {y}}^2w_{1}
		=0,\\
		&\partial_{ {t}} q_{1}+u_1\partial_{ {x}} q_{1}
		-2 q_{1}\frac{P_{t}+P_{x}u_{1}}{ Q}
		-\frac{ 2\gamma(P- q_{1}) q_{1}}{ Q}\partial_{ {x}}u_1
		+
		\sigma (\partial_{ {y}} q_{1})^{2}
		-2\sigma\frac{ \gamma(P- q_{1}) q_{1}}{ Q}\partial_{y}^2  q_{1}=0,
		\\
		&(u_1, w_1, \partial_{ {y}}q_{1})|_{ {y}=0}=\mathbf{0}, \\ &\lim_{ {y}\rightarrow +\infty}(u_1, w_{1}, q_1)( {t}, {x}, {y})=(U,I,H^2/2)( {t}, {x}),\\
		&(u_{1},w_{1}, q_{1})|_{ {t}=0}=(u_{1,0},w_{1,0}, h_{1,0}^2/2)( {x}, {y}).
	\end{aligned}
	\label{nohat}
	\right.
\end{equation}
Define matrix ${\bf{A}}_{0}, {\bf{B}}_{0}$ as
\begin{align*}
	{\bf{A}}_{0}={\bf{A}}_{0}(u_{1}, w_{1}, q_{1})=\left(
	\begin{array}{ccc}
		u_{1} & 0 & -\mathcal{A}\\
		0 & u_{1} & 0\\
		-\frac{ 2\gamma(P-q_{1})q_{1}}{ Q}  & 0 & u_{1}
	\end{array}\right),
\end{align*}
and
\begin{align*}
	{\bf{B}}_{0}={\bf{B}}_{0}(u_{1}, w_{1}, q_{1})=2q_{1}\left(
	\begin{array}{ccc}
		\mu\mathcal{A} & 0 & 0\\
		0 & \mu'\mathcal{A} & 0\\
		0 & 0 & \sigma\frac{ \gamma(P-q_{1})}{ Q}
	\end{array}\right),
\end{align*}
then the first 
three equations in \eqref{nohat} can be written as follows
\begin{align}
	\begin{split}
		&\partial_{t}
		\left(
		\begin{array}{ccc}
			u_{1}\\
			w_{1} \\
			q_{1} 
		\end{array}
		\right)
		+{\bf{A}}_{0}\partial_{x}
		\left(
		\begin{array}{ccc}
			u_{1}\\
			w_{1} \\
			q_{1} 
		\end{array}
		\right)
		+
		\left(
		\begin{array}{ccc}
			\left( \sigma-\mu\mathcal{A}\right) \partial_{ {y}}q_{1} \partial_{ {y}}u_{1} \\
			2\zeta\mathcal{A} \sqrt{2q_{1}} \partial_{ {y}}u_{1}
			+\left( \sigma-\mu'\mathcal{A}\right)   \partial_{ {y}}q_{1} \partial_{ {y}}w_{1} \\
			\sigma (\partial_{ {y}}q_{1})^{2} 
		\end{array}
		\right)
		\\
		&\qquad \qquad \quad	+
		\left(
		\begin{array}{ccc}
			\mathcal{A}P_{ {x}}\\
			0 \\
			-2q_{1}\frac{P_{t}+P_{x}u_{1}}{ Q} 
		\end{array}
		\right)
		-
		{\bf{B}}_{0}\partial_{y}^{2}
		\left(
		\begin{array}{ccc}
			u_{1}\\
			w_{1} \\
			q_{1} 
		\end{array}
		\right)=0.
	\end{split}
	\label{matrix}
\end{align}
We introduce a positive symmetric matrix
\begin{align}
	{\bf{S}}={\bf{S}}(u_{1}, w_{1}, q_{1})
	=\left(
	\begin{array}{ccc}
		\mathcal{A}^{-1}(P-q_{1})q_{1} & 0 & 0\\
		0 & \mathcal{A}^{-1}(P-q_{1})q_{1}& 0\\
		0 & 0 & \frac{ Q}{ 2\gamma}
	\end{array}\right),
	\label{S}
\end{align}
such that
\begin{align}
	{\bf{A}}={\bf{A}}(u_{1}, w_{1}, q_{1})=	{\bf{SA}}_{0}
	=\left(
	\begin{array}{ccc}
		\mathcal{A}^{-1}(P-q_{1})q_{1}u_{1} & 0 & -(P-q_{1})q_{1}\\
		0 & \mathcal{A}^{-1}(P-q_{1})q_{1}u_{1} & 0\\
		-(P-q_{1})q_{1}  & 0 &\frac{ Q}{ 2\gamma} u_{1}
	\end{array}\right),
	\label{A}
\end{align}
is symmetric and 
\begin{align}
	{\bf{B}}={\bf{B}}(u_{1}, w_{1}, q_{1})=	{\bf{SB}}_{0}=\left(
	\begin{array}{ccc}
		2\mu (P-q_{1})q_{1}^2 & 0 & 0\\
		0 & 2\mu'(P-q_{1}) q_{1}^2 & 0\\
		0 & 0 & \sigma (P-q_{1})q_{1}
	\end{array}\right)
	\label{B}
\end{align}
is positively definite. Finally, the system \eqref{nohat} by the equation  \eqref{matrix} is converted to 
\begin{equation}
	\left\{
	\begin{aligned}
		& \mathcal{A}^{-1}(P-q_{1})q_{1} \partial_{ {t}}u_1
		+\left\lbrace \mathcal{A}^{-1}(P-q_{1})q_{1}\right\rbrace u_1\partial_{ {x}}u_1
		+ (P-q_{1})q_{1} P_{ {x}}
		- (P-q_{1})q_{1} \partial_{ {x}}q_{1}
		\\
		&\quad
		+ \left( \mathcal{A}^{-1}\sigma-\mu\right) (P-q_{1})q_{1}  \partial_{ {y}}q_{1} \partial_{ {y}}u_{1}
		 -2\mu (P-q_{1})q_{1}^2\partial_{ {y}}^2u_1=0,\\
		& \mathcal{A}^{-1}(P-q_{1})q_{1}\partial_{ {t}}w_1
		+\left\lbrace \mathcal{A}^{-1}(P-q_{1})q_{1}\right\rbrace u_1\partial_{ {x}}w_1+2\zeta (P-q_{1})q_{1} \sqrt{2q_{1}} \partial_{ {y}}u_{1}
		\\
		&\quad
		+\left( \mathcal{A}^{-1}\sigma-\mu'\right) (P-q_{1})q_{1}   \partial_{ {y}}q_{1} \partial_{ {y}}w_{1}-2\mu' (P-q_{1})q_{1}^2 \partial_{ {y}}^2w_{1}
		=0,\\
		&\frac{ Q}{ 2\gamma}\partial_{ {t}}q_{1}
		+\frac{ Q}{ 2\gamma}u_1\partial_{ {x}}q_{1}
		-q_{1}\frac{P_{t}+P_{x}u_{1}}{ \gamma}
		-(P-q_{1})q_{1}\partial_{ {x}}u_1
	+
		 \frac{ \sigma Q}{ 2\gamma}(\partial_{ {y}}q_{1})^{2}
		-\sigma(P-q_{1})q_{1}\partial_{y}^2 q_{1}=0,
		\\
		&(u_1, w_1, \partial_{ {y}}q_{1})|_{ {y}=0}=\mathbf{0}, \\ &\lim_{ {y}\rightarrow +\infty}(u_1, w_{1}, q_1)( {t}, {x}, {y})=(U,I,H^2/2)( {t}, {x}),\\
		&(u_{1},w_{1}, q_{1})|_{ {t}=0}=(u_{1,0},w_{1,0}, h_{1,0}^2/2)( {x}, {y}).
	\end{aligned}
	\right.
	\label{3m}
\end{equation}

Now, we state the main result in this section.

\begin{theorem}\label{local}
Let $k \geq 4$ be an integer, we suppose that the trace  $(U, I, H, P)(t, x) $  of the outflow satisfies
\begin{align}
\sup_{0\leq t\leq T}\sum_{j=0}^{2k}
\|\partial^j_t (U, I, H, P)(t, \cdot) \|_{H^{2k-j}(\mathbb{T}_{x})}\leq M_e,
\label{me}
\end{align}
for  some constants $M_{e} >1 $. Assume that the initial data $(u_{1,0},w_{1,0}, h_{1,0}^2/2)( {x}, {y})$ satisfy
\begin{align}
(u_{1,0},w_{1,0}, h_{1,0}^2/2)( {x}, {y})- (U,I,H^2/2)( 0, {x}) \in H^{3k}(\Omega),
\label{v10}
\end{align}
and the compatibility conditions up to the $k$-$th$
order for the initial-boundary problem \eqref{3m}. Moreover, we further assume that the constant $\delta$ is chosen small enough
such that 
\begin{align}
 2\delta \leq \frac{1}{2}\big(h_{1,0}(x,y)\big)^2 \leq P(0,x)-2\delta, 
\qquad \forall (x,y)\in \Omega.
	\label{assumptions0}
\end{align}
Then there exists a positive time $T_{*} \leq T$ such that the 
 problem \eqref{3m} admits a unique solution $(u_{1}, w_{1}, q_{1})$ satisfying
\begin{align}
	(u_{1},w_{1}, h_{1}^2/2)(t, {x}, {y})- (U,I,H^2/2)( t, {x}) \in \mathcal{H}^{k}(\Omega_{T_{*}}),
	\label{v1}
\end{align}
and 
\begin{align}
	2\delta \leq \frac{1}{2}\big(h_{1}(t, x,y)\big)^2 \leq P(t,x)-2\delta, 
	\qquad \forall (t, x,y)\in \Omega_{T_{*}}.
	\label{assumptions}
\end{align}
\end{theorem}
\begin{remark}\label{rm}
Recalling the definition of matrix ${\bf{S}}$ given in \eqref{S} and the relation \eqref{qh}, it is easy to see that the Theorem \ref{local} is also true for  $(u_{1}, w_1, q_1)(t,x, y)$ of the problem \eqref{nohat} and $(\hat{u}_{1}, \hat{w}_1, \hat{h}_1)(\bar{t}, \bar{x}, \bar{y})$ of the problem  \eqref{hat}.
\end{remark}

Now, we use the classical Picard iteration scheme  to prove Theorem \ref{local}. To overcome the difficulty
originated from the boundary term, we introduce a smooth function $\phi(y)$ satisfying $0 \leq \phi(y) \leq 1$ and  
\begin{align*}
	\phi(y) =
	\begin{cases} 
		0, & 0 \leq y \leq 1, \\
		1, & y \geq 2.
	\end{cases}
\end{align*}
Write 
\begin{align*}
	u_{1}=u+U(t,x)\phi(y), \qquad w_{1}=w+I(t,x), \qquad  q_{1}=q+H^{2}(t,x)/2,
\end{align*}
then $(u, w, q)$ solves the following  equations from Equ. $\eqref{3m}_{1}-\eqref{3m}_{3}$:
\begin{equation}
	\left\{
	\begin{aligned}
		&	\mathcal{A}^{-1}(P-q_{1})q_{1}\partial_{ {t}}u
		+\left\lbrace 	\mathcal{A}^{-1}(P-q_{1})q_{1}\right\rbrace u_{1}\partial_{ {x}}u+
		\left\lbrace 	\mathcal{A}^{-1}(P-q_{1})q_{1}\right\rbrace uU_{x}\phi+
			(P-q_{1})q_{1} P_{ {x}}-(P-q_{1})q_{1}\partial_{ {x}} q
		\\
		&\qquad
		+\left( \sigma\mathcal{A}^{-1}-\mu\right) (P-q_{1})q_{1} \partial_{ {y}} q \partial_{ {y}}(u+U\phi)
		 -2\mu(P-q_{1})q_{1}^2
	 \partial_{ {y}}^2u
		-2\mu(P-q_{1})q_{1}^2
		 U\phi''\\
		&\quad=-	\mathcal{A}^{-1}(P-q_{1})q_{1}U_{t}\phi-	\mathcal{A}^{-1}(P-q_{1})q_{1}UU_{x}\phi^{2}
		+(P-q_{1})q_{1}HH_{x},\\
		&\mathcal{A}^{-1}(P-q_{1})q_{1}\partial_{ {t}}w
		+\left\lbrace \mathcal{A}^{-1}(P-q_{1})q_{1}\right\rbrace u_{1}\partial_{ {x}}w
		+\left\lbrace \mathcal{A}^{-1}(P-q_{1})q_{1}\right\rbrace uI_{x}
		+2\zeta (P-q_{1})q_{1} \sqrt{2q_{1}} \partial_{ {y}}u
		\\
		&\qquad
		+2\zeta(P-q_{1})q_{1} \sqrt{2q_{1}} U\phi'
		+\left(\sigma\mathcal{A}^{-1}-\mu'\right) (P-q_{1})q_{1} \partial_{ {y}} q \partial_{ {y}}w-2\mu'(P-q_{1})q_{1}^2
		 \partial_{ {y}}^2w\\
		&\quad=-\mathcal{A}^{-1}(P-q_{1})q_{1}I_{t}-\mathcal{A}^{-1}(P-q_{1})q_{1}UI_{x}\phi \\
		&\frac{ Q}{ 2\gamma}\partial_{ {t}} q
		+\frac{ Q}{ 2\gamma}u_{1}\partial_{ {x}} q
		+\frac{ Q}{ 2\gamma}uHH_{x}
		-q_{1}\frac{P_{t}+P_{x}(u+U\phi)}{\gamma}
		-(P-q_{1})q_{1}\partial_{ {x}}u
		-(P-q_{1})q_{1}U_{x}\phi
		\\
		&\qquad
		+
		\sigma\frac{ Q}{ 2\gamma}(\partial_{ {y}} q)^{2}
		-\sigma(P-q_{1})q_{1}\partial_{y}^2  q
		\\
		&
		\quad=-\frac{ Q}{ 2\gamma}HH_{t}	-		\frac{ Q}{ 2\gamma}UHH_{x}\phi.
	\end{aligned}
	\right.
\end{equation}
The corresponding 
boundary conditions  and the initial data in the system \eqref{3m} become
\begin{equation}
	\left\{
	\begin{aligned}
		&(u, w, \partial_{ {y}}q)|_{ {y}=0}=\mathbf{0}, \\ &\lim_{ {y}\rightarrow +\infty}(u_1, w_{1}, q_1)( {t}, {x}, {y})=\mathbf{0},\\
		&(u, w, q)|_{ {t}=0}=(u_{1,0},w_{1,0}, h_{1,0}^2/2)( {x}, {y})-
		(U(0,x)\phi(y), I(0,x), H^{2}(0,x)/2).
	\end{aligned}
	\right.
\end{equation}

Let ${\bf{v}}={\bf{v}}(t,x, y)=: (u, w, q)^{T}(t, x, y)$ and ${\bf{v}}_{1}(t,x, y)={\bf{v}}+(U(t,x)\phi(y), I(t,x), H^{2}(t,x)/2)^{T}$,
we obtain ${\bf{v}}$ by solving the following
linear initial-boundary value problem
\begin{equation}\label{vvv}
	\left\{
	\begin{aligned}
		&
			{\bf{S}}({\bf{v}}_{1})\partial_{t}{\bf{v}}+{\bf{A}}({\bf{v}}_{1})\partial_{x}{\bf{v}}+f({\bf{v}}_{1}, \partial_{y}{\bf{v}}_{1})
		-{\bf{B}}({\bf{v}}_{1})\partial_{y}^{2}{\bf{v}}
		=g({\bf{v}}_{1})\\
		&(u, w, \partial_{ {y}}q)|_{ {y}=0}=\mathbf{0}, \\ &\lim_{ {y}\rightarrow +\infty}{\bf{v}}( {t}, {x}, {y})=\mathbf{0},\\
		&{\bf{v}}|_{ {t}=0}=(u_{1,0},w_{1,0}, h_{1,0}^2/2)^{T}( {x}, {y})-
		(U(0,x)\phi(y), I(0,x), H^{2}(0,x)/2)^{T},
	\end{aligned}
	\right.
\end{equation}
with the  matrices
\begin{align}
\begin{split}
	f({\bf{v}}_{1}, \partial_{y}{\bf{v}}_{1})
	&=\left(
	\begin{array}{ccc}
		\left( \sigma\mathcal{A}^{-1}-\mu\right) (P-q_{1})q_{1}\partial_{ {y}} q \partial_{ {y}}u_{1} \\
		2\zeta(P-q_{1})q_{1} \sqrt{2q_{1}} \partial_{ {y}}u_{1}
		+\left(  \sigma\mathcal{A}-^{-1}\mu'\right)(P-q_{1})q_{1}  \partial_{ {y}} q \partial_{ {y}}w \\
		\sigma\frac{ Q}{ 2\gamma}(\partial_{ {y}} q)^{2}
	\end{array}\right)\\
	&	=\left(
	\begin{array}{ccc}
		\left( \sigma\mathcal{A}^{-1}-\mu\right)(P-q_{1})q_{1} \partial_{ {y}} q & 0 & 0 \\
		2\zeta(P-q_{1})q_{1} \sqrt{2q_{1}} & 
		\left(  \sigma\mathcal{A}^{-1}-\mu'\right) (P-q_{1})q_{1}  \partial_{ {y}} q & 0 \\
		0 & 0 & \sigma\frac{ Q}{ 2\gamma}\partial_{ {y}} q
	\end{array}\right)\partial_{y}{\bf{v}}_{1}
	\\
	&={\bf{F}}({\bf{v}}_{1}, \partial_{y}{\bf{v}}_{1})\partial_{y}{\bf{v}}_{1},
	\end{split}
\label{f}
\end{align}
and
\begin{align}
	g({\bf{v}}_{1})
	=\left(
	\begin{array}{ccc}
	(P-q_{1})q_{1}\left( 	-\mathcal{A}^{-1}uU_{x}\phi- P_{ {x}}+HH_{x}-2\mu
		\left(q+\frac{H^2}{2}\right) U\phi''
		-\mathcal{A}^{-1}U_{t}\phi-\mathcal{A}^{-1}UU_{x}\phi^{2} \right) \\
		-\mathcal{A}^{-1}	(P-q_{1})q_{1}\left( uI_{x}-I_{t}-UI_{x}\phi \right) \\
		-\frac{ Q}{ 2\gamma}uHH_{x}
		+q_{1}\frac{P_{t}+P_{x}(u+U\phi)}{\gamma}
		+	(P-q_{1})q_{1}U_{x}\phi
		-\frac{ Q}{ 2\gamma}HH_{t}-\frac{ Q}{ 2\gamma}UHH_{x}\phi
	\end{array}\right),
\label{g}
\end{align}
where the other matrices ${\bf{S}}(\cdot)$, ${\bf{A}}(\cdot)$ and ${\bf{B}}(\cdot)$ are defined in \eqref{S}, \eqref{A} and \eqref{B}, respectively.

Next, set
\begin{align}\label{m11}
{\bf{v}}_{1,0}^{j}(x, y)=	(u_{1,0}^{j}, w_{1,0}^j , b_{1,0}^j )^{T}(x,y)=(\partial_t^j u_{1}, \partial_t^j w_{1}, \partial_t^j b_{1})^{T}(0,x,y),\quad 0\leq j\leq k.
\end{align}
Under the assumptions in Theorem \ref{local}, it follows that ${\bf{v}}_{1,0}^{j}(x, y)\in {H}^m(\Omega)$, which can be derived from initial-boundary value problem \eqref{vvv} by induction with respect to $j$.  Moreover,  we also obtain 
\begin{align}\label{m12}
	\sup_{0\leq t \leq T}\sum_{\alpha_1+\alpha_2\leq 2k}\left\|\partial^\alpha (U\phi, I, H^{2}/2)(t,\cdot)\right\|_{L^{2}(\mathbb{T}_x, L^\infty(\mathbb{R}_y^+))}
		\leq M_1,
\end{align}
which implies that there exists a positive constant $M_1>1$, depending only on $M_e$ and $(u_{1,0},w_{1,0}, h_{1,0}^2/2)^{T}( {x}, {y})-
(U(0,x)\phi(y), I(0,x), H^{2}(0,x)/2)^{T}$, such that
\begin{align}\label{m13}
	\sum_{j=0}^{k}\big\|{\bf{v}}_{1,0}^{j}(x, y)-\partial_{t}^{j}(U(0,x)\phi(y), I(0,x), H^{2}(0,x)/2)^{T}\big\|_{H^{3k-2j}(\Omega)}\leq M_1.
\end{align}

\subsection{Construction of approximate solution sequence}\label{s3.1}
In this subsection, let us construct a sequence of approximate solution $\{\vn\}_{n \geq 0}=\{(u^{n}, w^{n}, q^{n})\}_{n \geq 0}$ through the following procedure.

We begin by constructing the zero-th approximate solution ${\bf{v}}^0(t, x, y) = (u^0, w^0, q^0)^\top(t, x, y)$ for problem  \eqref{vvv}. This approximation needs to  satisfy:
\begin{enumerate}
	\item
	boundary conditions:
	\begin{equation}\label{eq:bc}
		(u^0, w^0, \partial_y q^0)|_{y=0} = \mathbf{0},  \quad 
		\lim_{y \to +\infty} {\bf{v}}^0(t, x, y) = 0;
	\end{equation}
	\item  compatibility conditions up to the $k$-th order:
	\begin{equation}\label{eq:compatibility}
		\partial_t^j {\bf{v}}^0|_{t=0} = {\bf{v}}_{0}^j(x, y), \quad 0 \leq j \leq k.
	\end{equation}
\end{enumerate}
Such an approximation can indeed be constructed, we define ${\bf{v}}^0$ explicitly as:
\begin{align}
	{\bf{v}}^0(t, x, y) :
=(u_{1}^0,w_{1}^0, (h_{1}^{0})^2/2)^{T}(t, {x}, {y})-
(U(t,x)\phi(y), I(t,x), H^{2}(t,x)/2)^{T}
	= \sum_{j=0}^k \frac{t^j}{j!}{\bf{v}}_{0}^{j}(x, y),
\end{align}
where ${\bf{v}}_{0}^{j}(x, y)={\bf{v}}_{1,0}^{j}(x, y)-\partial_{t}^{j}(U(0,x)\phi(y), I(0,x), H^{2}(0,x)/2)^{T}$.
It is straightforward to check, by \eqref{m13},  that
\begin{align}
	{\bf{v}}^0(t, x, y) \in H^{k}({\Omega_{T}}).
\end{align}
Moreover, by an argument similar to Proposition 2.1  in \cite{HLY2019}, we can deduce that there exists a  time $0 <T_{0} \leq T$, such that
\begin{align}\label{0intial}
\delta\leq	q^{0}(t, x, y)+H^2/2(t,x) \leq P(t, x)-\delta,\quad \forall\ (t,x,y)\in\Omega_{T_0},
\end{align}
with $\delta>0$ given in  Theorem \ref{local}.

Next, we construct the $n$-th order  approximate solution by induction. Precisely,
for the problem \eqref{vvv} in the region $\Omega_{T_1},\ T_1\in(0,T_0]$  with $T_1$ to be determined later, suppose that
we have been constructed the $(n-1)$-th order approximate solution ${\bf{v}}^{n-1}(t, x, y)=(u^{n-1}, w^{n-1}, q^{n-1})^{T}(t, x, y)$, which satisfies
\begin{align}
	{\bf{v}}^{n-1}(t, x, y)\in \mathcal{H}^k(\Omega_{T_1}),\quad k\geq 4,
\label{3.1}
\end{align}
and  the following induction hypotheses
\begin{equation}
	\left\{\begin{aligned}
		&\partial_{t}^{j} {\bf{v}}^{n-1}\big|_{y=0}={\bf{v}}_{0}^{j}(x, y), \qquad 0\leq j \leq k,\\
	&(u^{n-1}, w^{n-1}, \partial_{y}q^{n-1})\big|_{y=0}=\mathbf{0}, \\
	&\lim_{ {y}\rightarrow +\infty}{\bf{v}}^{n-1}( {t}, {x}, {y})=\mathbf{0}, \\
	&\delta \leq q^{n-1}(t,x,y)+ H^{2}(t,x)/2\leq P(t,x)-\delta.
 	\end{aligned}
	\right.
\label{3.2}
\end{equation}
Finally, we are going to construct the $n$-th order approximate solution
\begin{align*}
{\bf{v}}^{n}={\bf{v}}^{n}(t, x, y)=(u^n, w^n, q^n)^{T}(t, x, y)
\end{align*}
 to the problem \eqref{vvv} that satisfies the corresponding conditions to \eqref{3.1}- \eqref{3.2}. Actually, we construct 
 ${\bf{v}}^{n}=(u^n, w^n, q^n)^{T}$ 
 by solving the following linear initial-boundary value problem in $\Omega_{T_1}$,
\begin{equation}\label{3.00}
	\left\{
	\begin{aligned}
		&{\bf{S}}({\bf{v}}^{n-1}_{1})\partial_{t}{\bf{v}}^{n}
		+\an\partial_{x}\vn+\fn\partial_{y}\vn_{1}
		-\bn\partial_{y}^{2}\vn=g(\vnn_{1}),\\
		&(u^{n}, w^{n}, \partial_{y}q^{n})\big|_{y=0}=\mathbf{0},
		\qquad 
		\lim_{ {y}\rightarrow +\infty}{\bf{v}}^{n}( {t}, {x}, {y})=\mathbf{0},\\
		&{\bf{v}}^{n}|_{ {t}=0}=(u_{1,0},w_{1,0}, h_{1,0}^2/2)^{T}( {x}, {y})-
		(U(0,x)\phi(y), I(0,x), H^{2}(0,x)/2),
	\end{aligned}
	\right.
\end{equation}
where ${\bf{S}}(\cdot)$, ${\bf{A}}(\cdot)$, ${\bf{B}}(\cdot)$, ${\bf{F}}(\cdot)$ and $g(\cdot)$ are defined in \eqref{S}, \eqref{A}, \eqref{B}, \eqref{f} and \eqref{g}, respectively.

\subsection{Uniform estimates of $\vn$}
The objective of this subsection is to derive  the uniform energy estimates of approximate solution $\vn$ for the system \eqref{3.00}.
Indeed, we will prove the solvability of the above problem \eqref{3.00}.
\begin{lemma}\label{solvability}
Under the assumptions of Theorem \ref{local}, the problem \eqref{3.00} has a unique classical solution $ \vn(t,x,y) \in \mathcal{H}^k(\Omega_{T_1})$  satisfying 
	\begin{align}
		\partial^j_t  \vn|_{t=0}= {\bf{v}}_0^j(x,y), \qquad 0\leq j\leq k.
	\label{initial}
	\end{align}
Moreover, there exists a positive constant $C_0$,  depending only on $T_0, M_e, \delta, \gamma$ 
such that	the following estimate holds for $t \in [0, T_{1}]$:
\begin{align}
	&\left\|\vn\right\|_{\mathcal{H}^k(\Omega_{t})}^2
	+\int_{0}^{t}\left\| \partial_{y}\vn\right\|_{\mathcal{H}^{k}(\Omega)}^{2} d\tau
	\leq
	C_{0}M_{1}^{3}\exp \left( C_{0}\int_{0}^{t}\left\| \vnn\right\|_{\mathcal{H}^k(\Omega)}^{10} ~d\tau
	\right),
\label{ginq}
\end{align}
	where the constant $M_1\geq 1$ is given in \eqref{m13}.
\end{lemma}
\begin{proof}
The proof of local existence and uniqueness  of the solution $\vn$ for the linear problem \eqref{3.00} follows by standard method through uniform  estimates \eqref{ginq} combined with compactness argument.  Therefore, our primary focus lies in rigorously establishing the uniform energy estimate \eqref{ginq}.	
As for \eqref{initial}, it is directly derived from a straightforward computation by induction hypotheses $\eqref{3.2}_{1}$.

Applying the differential operator $\partial^{\alpha} = \partial_{t}^{\alpha_1} \partial_{x}^{\alpha_2} \partial_{y}^{\alpha_3}$ with $|\alpha| \leq k$ to both sides of the equations $\eqref{3.00}_{1}$,  and taking the $L^2$
inner product on the resulting
equation with $\partial^{\alpha}\vn$ to yield
\begin{align}
	\begin{split}
	0
&=\langle \partial^{\alpha}\vn, \partial^{\alpha}(	\sn\partial_{t}{\bf{v}}^{n})\rangle
+
\langle \partial^{\alpha}\vn, \partial^{\alpha}(	\an\partial_{x}\vn)\rangle
+\left\langle \partial^{\alpha}\vn, \partial^{\alpha}\left(  \fn\partial_{y}\vn_{1}\right) \right\rangle 
\\
&\quad -\left\langle \partial^{\alpha}\vn, \bn\partial^{\alpha}\partial_{y}^{2}\vn \right\rangle -\left\langle \partial^{\alpha}\vn, \left[ \partial^{\alpha}, \bn\right] \partial_{y}^{2}\vn	\right\rangle-\left\langle \partial^{\alpha}\vn, \partial^{\alpha}g(\vnn)\right\rangle 
=: \sum\limits_{i=1}^{6}K_i.
	\end{split}
\label{3.3}
\end{align}
Now, we estimate the right-hand side of \eqref{3.3} term by term as follows.

\textbf{Dealing with $K_1$  term :} For $K_{1}$, it follows from Lemma \ref{sobolevtype} and \eqref{inq3} that
\begin{align}
		K_1&
		=\left\langle \partial^{\alpha}\vn, \sn \partial^{\alpha}\partial_{t}\vn \right\rangle 
		+\left\langle  \partial^{\alpha}\vn, \left[ \partial^{\alpha}, \sn\right]\partial_{t}\vn \right\rangle \nonumber\\
		&=\frac{1}{2}\frac{d}{dt}\left\langle \partial^{\alpha}\vn, \sn \partial^{\alpha}\vn\right\rangle 
		-\frac{1}{2}\left\langle \partial^{\alpha}\vn, \partial_{t}\sn\partial^{\alpha}\vn\right\rangle 
		+\left\langle  \partial^{\alpha}\vn, \left[ \partial^{\alpha}, \sn\right]\partial_{t}\vn \right\rangle \nonumber\\
		&\geq \frac{1}{2}\frac{d}{dt}\left\langle \partial^{\alpha}\vn, \sn \partial^{\alpha}\vn\right\rangle 
		-C\left\| \partial^{\alpha}\vn\right\|_{L^2(\Omega)}\left( \left\| \partial_{t}\sn\right\|_{L^\infty(\Omega)}\left\| \partial^{\alpha}\vn\right\|_{L^2(\Omega)}+\left\| \sn\right\|_{\mathcal{H}^k(\Omega)}\left\| \partial_{t}\vn\right\| _{\mathcal{H}^{k-1}}\right) \nonumber\\
		&\geq \frac{1}{2}\frac{d}{dt}\left\langle \partial^{\alpha}\vn, \sn \partial^{\alpha}\vn\right\rangle 
		-C\left\| \vn\right\|_{\mathcal{H}^k(\Omega)}^{2}\left( 1+\left\| \vnn\right\|_{\mathcal{H}^3(\Omega)}^{4}+\left\| \vnn\right\|_{\mathcal{H}^k(\Omega)}^{3}\right) .
\label{k1}
\end{align}

\textbf{Dealing with $K_2$  term :}
Decompose $K_{2}$ into
\begin{align*}
	K_{2}=K_{2}^{1}+K_{2}^{2}
\end{align*}
with
\begin{align*}
	K_{2}^{1}=\left\langle \partial^{\alpha} \vn, \an\partial^{\alpha}\partial_{x}\vn
	\right\rangle ,
\end{align*}
and
\begin{align*}
	K_{2}^{1}=\left\langle \partial^{\alpha} \vn, \left[ \partial^{\alpha},\an\right] \partial_{x}\vn
	\right\rangle .
\end{align*}
Using the  integration by parts and  Lemma \ref{sobolevtype}, $K_{2}^{1}$ can be estimated as follows.
\begin{align}
	K_{2}^{1} 
	&\leq \frac{1}{2}\left| \left\langle \partial^{\alpha} \vn, \partial_{x}\an\partial^{\alpha}\vn
	\right\rangle\right| \nonumber\\
	& \leq C\left\| \partial^{\alpha}\vn\right\|_{L^{2}(\Omega)}^{2}\left\| \partial_{x}\an\right\|_{L^\infty(\Omega)} \nonumber\\
	&\leq
	C\left\| \vn\right\|_{\mathcal{H}^k(\Omega)}^{2}\left( 1+\left\| \vnn\right\|_{\mathcal{H}^3(\Omega)}^{5}\right).\label{3.5}
\end{align}
Similarly, by \eqref{inq3}, we also have 
\begin{align}
	K_{2}^{2} 
	&\leq C \left\| \partial^{\alpha}\vn\right\|_{L^{2}(\Omega)}\left\| \left[ \partial^{\alpha},\an\right] \partial_{x}\vn\right\|_{L^{2}(\Omega)}\nonumber\\
	&\leq C \left\| \partial^{\alpha}\vn\right\|_{L^{2}(\Omega)}
	\left\| \an\right\|_{H^{k}(\Omega)}
	\left\|  \partial_{x}\vn\right\|_{H^{k-1}(\Omega)} \nonumber\\
	&\leq  	C\left\| \vn\right\|_{\mathcal{H}^k(\Omega)}^{2}\left( 1+\left\| \vnn\right\|_{\mathcal{H}^4(\Omega)}^{4}\right).
\label{3.6}
\end{align}
Combining estimates \eqref{3.5} and \eqref{3.6}, we obtain
\begin{align}
	K_{2}
	\leq  	C\left\| \vn\right\|_{\mathcal{H}^k(\Omega)}^{2}\left( 1+\left\| \vnn\right\|_{\mathcal{H}^k(\Omega)}^{5}\right).
	\label{k2}
\end{align}

\textbf{Dealing with $K_3$  term :}
To estimate $K_{3}$, 
we divide $K_{3}$ into the following three parts.
\begin{align}
\begin{split}
K_{3}
&=\left\langle \partial^{\alpha}\vn,  \fn\partial^{\alpha}\partial_{y}\vn_{1} \right\rangle 
+
\left\langle \partial^{\alpha}\vn,  {{\bf{F}}({\bf{v}}_{1}^{n-1}, \partial^{\alpha}\partial_{y}{\bf{v}}_{1}^{n-1})} \partial_{y}\vn_{1} \right\rangle  \\
&\quad+
\sum\limits_{ \alpha' \leq \alpha, 1\leq |\alpha'|}
\sum\limits_{ 0\leq \beta' \leq \alpha',  |\beta'| \leq |\alpha'|-1}
C^{\alpha'}_\alpha C^{\beta'}_{\alpha'}
\left\langle \partial^{\alpha}\vn, 
{{\bf{F}}(\partial^{\alpha'-\beta'}{\bf{v}}_{1}^{n-1}, \partial^{\beta'}\partial_{y}{\bf{v}}_{1}^{n-1})}
\partial^{\alpha-\alpha'}\partial_{y}\vn_{1} \right\rangle .
\end{split}
\label{3.7}
\end{align}
For the first term, a direct calculation  leads to
\begin{align}
&\left\langle \partial^{\alpha}\vn,  \fn\partial^{\alpha}\partial_{y}\vn_{1} \right\rangle \nonumber\\
&\leq
\left\| \partial^{\alpha}\vn\right\|_{L^2(\Omega)} \left\| \fn\right\|_{L^\infty(\Omega)}\left\| \partial^{\alpha}\partial_{y}\left(\vn+(U\phi, I, H^{2}/2)\right) \right\|_{L^2(\Omega)} \nonumber\\
&\leq C\left\| \vn\right\|_{\mathcal{H}^k(\Omega)}\left( 1+\left\| \vnn\right\|_{\mathcal{H}^3(\Omega)}^{5}\right)\left(1+\left\| \partial_{y}\vn\right\|_{\mathcal{H}^k(\Omega)}\right). 
\label{3.8}
\end{align}
For the second term, integration by parts leads to
\begin{align}
\begin{split}
&\left\langle \partial^{\alpha}\vn,  {{\bf{F}}({\bf{v}}_{1}^{n-1}, \partial^{\alpha}\partial_{y}{\bf{v}}_{1}^{n-1})} \partial_{y}\vn_{1} \right\rangle
\\
&=\int_{\mathbb{T}} \partial^{\alpha}\vn~ {{\bf{F}}({\bf{v}}_{1}^{n-1}, \partial^{\alpha}{\bf{v}}_{1}^{n-1})} \partial_{y}\vn_{1} \big|_{y=0}dx-
\left\langle \partial^{\alpha}\partial_{y}\vn,  {{\bf{F}}({\bf{v}}_{1}^{n-1}, \partial^{\alpha}{\bf{v}}_{1}^{n-1})} \partial_{y}\vn_{1} \right\rangle
\\
&\quad-
\left\langle \partial^{\alpha}\vn,  {{\bf{F}}({\bf{v}}_{1}^{n-1}, \partial^{\alpha}{\bf{v}}_{1}^{n-1})} \partial_{y}^{2}\vn_{1} \right\rangle
-\left\langle \partial^{\alpha}\vn,  {{\bf{F}}(\partial_{y}{\bf{v}}_{1}^{n-1}, \partial^{\alpha}{\bf{v}}_{1}^{n-1})} \partial_{y}\vn_{1} \right\rangle.
\end{split}
\label{3.9}
\end{align}
Notice from the last three terms on the right-hand side of the equation \eqref{3.9} that the order of $\partial^{\alpha}\vnn$ cannot exceed $k$. Then thanks to \eqref{inq1} and  H\"{o}lder's inequality, we
get, by a similar derivation of \eqref{3.8}, that
\begin{align}
&\left| \left\langle \partial^{\alpha}\partial_{y}\vn,  {{\bf{F}}({\bf{v}}_{1}^{n-1}, \partial^{\alpha}{\bf{v}}_{1}^{n-1})} \partial_{y}\vn_{1} \right\rangle
+
\left\langle \partial^{\alpha}\vn,  {{\bf{F}}({\bf{v}}_{1}^{n-1}, \partial^{\alpha}{\bf{v}}_{1}^{n-1})} \partial_{y}^{2}\vn_{1} \right\rangle
+\left\langle \partial^{\alpha}\vn,  {{\bf{F}}(\partial_{y}{\bf{v}}_{1}^{n-1}, \partial^{\alpha}{\bf{v}}_{1}^{n-1})} \partial_{y}\vn_{1} \right\rangle\right| 
\nonumber\\
&\leq C\left( \left\| \vn\right\|_{\mathcal{H}^k(\Omega)}+\left\| \partial_{y}\vn\right\|_{\mathcal{H}^k(\Omega)} \right) 
\left( \left\| \partial_{y}\vn_{1}\right\|_{L^\infty(\Omega)}
+
\left\| \partial_{y}^{2}\vn_{1}\right\|_{L^\infty(\Omega)}\right)
\nonumber \\
&\qquad \cdot
\left( 
\left\| {{\bf{F}}({\bf{v}}_{1}^{n-1},\cdot)} \right\|_{\mathcal{H}^k(\Omega)}
\left\|{{\bf{F}}(\cdot, {\bf{v}}_{1}^{n-1})} 
 \right\|_{\mathcal{H}^k(\Omega)}
 +
 \left\| {{\bf{F}}(\partial_{y}{\bf{v}}_{1}^{n-1},\cdot)} \right\|_{L^\infty(\Omega)}
 \left\|{{\bf{F}}(\cdot, \partial^{\alpha}{\bf{v}}_{1}^{n-1})} 
 \right\|_{L^2(\Omega)}
 \right)
 \nonumber\\
 &\leq C \left( \left\| \vn\right\|_{\mathcal{H}^k(\Omega)}+\left\| \partial_{y}\vn\right\|_{\mathcal{H}^k(\Omega)} \right) 
 \left(1+ \left\| \vn\right\|_{\mathcal{H}^4(\Omega)} \right) 
  \left(1+ \left\| \vnn\right\|_{\mathcal{H}^k(\Omega)}^{5}\right),
  \label{3.10}
\end{align}
where the matrix ${{\bf{F}}(\partial^{\beta_{1}}{\bf{v}}_{1}^{n-1}, \partial^{\beta_{2}}{\bf{v}}_{1}^{n-1})}$ is formally decomposed into the product of two related matrices ${{\bf{F}}_{1}(\partial^{\beta_{1}}{\bf{v}}_{1}^{n-1}, \cdot)}\cdot{{\bf{F}}_{2}(\cdot, \partial^{\beta_{2}}{\bf{v}}_{1}^{n-1})}$. Combining \eqref{3.9} and \eqref{3.10}, we arrive at
\begin{align}
\begin{split}
	&-\left\langle \partial^{\alpha}\vn,  {{\bf{F}}({\bf{v}}_{1}^{n-1}, \partial^{\alpha}\partial_{y}{\bf{v}}_{1}^{n-1})} \partial_{y}\vn_{1} \right\rangle
\\
&=-\int_{\mathbb{T}} \partial^{\alpha}\vn~ {{\bf{F}}({\bf{v}}_{1}^{n-1}, \partial^{\alpha}{\bf{v}}_{1}^{n-1})} \partial_{y}\vn_{1} \big|_{y=0}dx
\\
&\quad+
C \left( \left\| \vn\right\|_{\mathcal{H}^k(\Omega)}+\left\| \partial_{y}\vn\right\|_{\mathcal{H}^k(\Omega)} \right) 
\left(1+ \left\| \vn\right\|_{\mathcal{H}^4(\Omega)} \right) 
\left(1+ \left\| \vnn\right\|_{\mathcal{H}^k(\Omega)}^{5}\right).
\end{split}
\label{3.11}
\end{align}
For the third term, on the one hand, when $\alpha'=\beta$, we have
\begin{align}
&\sum\limits_{ \alpha' \leq \alpha, 1\leq |\alpha'|}
\sum\limits_{ \beta \leq \alpha',  |\beta| \leq |\alpha'|-1}
C^{\alpha'}_\alpha C^{\beta}_{\alpha'}
\left\langle \partial^{\alpha}\vn, 
{{\bf{F}}(\partial^{\alpha'-\beta}{\bf{v}}_{1}^{n-1}, \partial^{\beta}\partial_{y}{\bf{v}}_{1}^{n-1})}
\partial^{\alpha-\alpha'}\partial_{y}\vn_{1} \right\rangle 
\nonumber\\
&=
\sum\limits_{ \beta \leq \alpha, 1\leq |\beta| \leq |\alpha|-1}
C^{\beta}_\alpha
\left\langle \partial^{\alpha}\vn, 
{{\bf{F}}({\bf{v}}_{1}^{n-1}, \partial^{\beta}\partial_{y}{\bf{v}}_{1}^{n-1})}
\partial^{\alpha-\beta}\partial_{y}\vn_{1} \right\rangle 
\nonumber\\
&
\leq
C\left\| \partial^{\alpha}\vn\right\|_{L^2(\Omega)}
\left\| {{\bf{F}}_{1}({\bf{v}}_{1}^{n-1}, \cdot)}\right\|_{L^\infty(\Omega)}
\left\| {{\bf{F}}_{2}(\cdot, \partial^{\beta}\partial_{y}{\bf{v}}_{1}^{n-1})}
\partial^{\alpha-\beta}\partial_{y}\vn_{1}\right\|_{L^2(\Omega)}
\nonumber\\
&
\leq C\left\| \vn\right\|_{\mathcal{H}^{k}(\Omega)}
\left( 1+\left\| {\bf{v}}^{n-1}\right\|_{\mathcal{H}^{k}(\Omega)}^{5}\right) 
\left( 1+\left\| \vn\right\|_{\mathcal{H}^{k}(\Omega)}\right),
\label{3.12}
\end{align}
where we have used the   following fact
\begin{align*}
\left\| {{{\bf{F}}_{2}}(\cdot, \partial^{\beta}\partial_{y}{\bf{v}}_{1}^{n-1})}
\partial^{\alpha-\beta}\partial_{y}\vn_{1}\right\|_{L^2(\Omega)}
&\leq
C
\left\{
\begin{aligned}
	&\left\| {{\bf{F}}_{2}(\cdot, \partial^{\beta}\partial_{y}{\bf{v}}_{1}^{n-1})}\right\|_{L^\infty(\Omega)}
	\left\| \partial^{\alpha-\beta}\partial_{y}\vn_{1}\right\|_{L^2(\Omega)}~~\beta=1, 2,\nonumber\\
	& \left\| {{\bf{F}}_{2}(\cdot, \partial^{\beta}\partial_{y}{\bf{v}}_{1}^{n-1})}\right\|_{L^2(\Omega)}
	\left\| \partial^{\alpha-\beta}\partial_{y}\vn_{1}\right\|_{L^\infty(\Omega)}~~3\leq\beta \leq k-1, \nonumber\\
\end{aligned}
\right.
\\
& \leq
C
\left( 1+\left\| {\bf{v}}^{n-1}\right\|_{\mathcal{H}^{k}(\Omega)}\right) 
\left( 1+\left\| \partial_{y}\vn\right\|_{\mathcal{H}^{k-1}(\Omega)}\right) .
\end{align*}
On the other hand, when $\alpha'-\beta \geq 1$, by \eqref{inq1}, we arrive at
\begin{align*}
&	\left\|
	{{\bf{F}}(\partial^{\alpha'-\beta}{\bf{v}}_{1}^{n-1}, \partial^{\beta}\partial_{y}{\bf{v}}_{1}^{n-1})}
	\partial^{\alpha-\alpha'}\partial_{y}\vn_{1}\right\|_{L^2(\Omega)}
	\nonumber\\
	&\leq
	C
	\left\{
	\begin{aligned}
		&\left\| {{\bf{F}}(\partial^{\alpha'-\beta}{\bf{v}}_{1}^{n-1}, \partial^{\beta}\partial_{y}{\bf{v}}_{1}^{n-1})}\right\|_{L^\infty(\Omega)}
		\left\| \partial^{\alpha-\alpha'}\partial_{y}\vn_{1}\right\|_{L^2(\Omega)}~~\alpha'=1, 2,\nonumber\\
		& \left\| {{\bf{F}}(\partial^{\alpha'-\beta}{\bf{v}}_{1}^{n-1}, \partial^{\beta}\partial_{y}{\bf{v}}_{1}^{n-1})}\right\|_{L^2(\Omega)}
		\left\| \partial^{\alpha-\alpha'}\partial_{y}\vn_{1}\right\|_{L^\infty(\Omega)}~~3\leq\alpha' \leq k, ~~ 0 \leq |\beta| \leq |\alpha'|-1, \nonumber\\
	\end{aligned}
	\right.
	\\
	&\leq
C
\left\{
\begin{aligned}
	&\left( 1+\left\| \vnn\right\|^5_{\mathcal{H}^{4}(\Omega)}\right)
	\left( 1+\left\| \partial_{y}\vn\right\|_{\mathcal{H}^{k-1}(\Omega)}\right) ,\nonumber\\
	& \left\| {{\bf{F}}(\partial^{\alpha'-\beta-1}\partial{\bf{v}}_{1}^{n-1}, \partial^{\beta}\partial_{y}{\bf{v}}_{1}^{n-1})}\right\|_{L^2(\Omega)}
	\left( 1+\left\| \partial_{y}\vn\right\|_{\mathcal{H}^{k-1}(\Omega)}\right),~~3\leq\alpha' \leq k, ~~ 0 \leq |\beta| \leq |\alpha'|-1, \nonumber\\
\end{aligned}
\right.
\\
	&\leq
C
\left\{
\begin{aligned}
&\left( 1+\left\| \vnn\right\|^5_{\mathcal{H}^{4}(\Omega)}\right)
\left( 1+\left\| \vn\right\|_{\mathcal{H}^{k}(\Omega)}\right) ,\nonumber\\
	& \left\| {{\bf{F}}_{1}({\bf{v}}_{1}^{n-1}, \cdot)}\right\|_{\mathcal{H}^{k-1}(\Omega)}
	\left\| {{\bf{F}}_{2}( \cdot, \partial_{y}{\bf{v}}_{1}^{n-1})}\right\|_{\mathcal{H}^{k-1}(\Omega)}
\left( 1+\left\| \partial_{y}\vn\right\|_{\mathcal{H}^{k-1}(\Omega)}\right), \nonumber\\
\end{aligned}
\right.
\\
	& \leq
	C
	\left( 1+\left\| {\bf{v}}^{n-1}\right\|_{\mathcal{H}^{k}(\Omega)}^{5}\right) 
	\left( 1+\left\| \vn\right\|_{\mathcal{H}^{k}(\Omega)}\right),
\end{align*}
where $\partial\vnn_{1}=\sum\limits_{|\alpha|=1}\partial^{\alpha}\vnn_{1}$.
With the above estimate, it is easy to check, for
$\alpha'-\beta \geq 1$, that 
\begin{align}
\begin{split}
	&\sum_{ \alpha' \leq \alpha, 1\leq |\alpha'|}
	\sum\limits_{ \beta \leq \alpha',  |\beta| \leq |\alpha'|-1}
	C^{\alpha'}_\alpha C^{\beta}_{\alpha'}
	\left\langle \partial^{\alpha}\vn, 
	{{\bf{F}}(\partial^{\alpha'-\beta}{\bf{v}}_{1}^{n-1}, \partial^{\beta}\partial_{y}{\bf{v}}_{1}^{n-1})}
	\partial^{\alpha-\alpha'}\partial_{y}\vn_{1} \right\rangle 
\\
	&
	\leq C\left\| \vn\right\|_{\mathcal{H}^{k}(\Omega)}
	\left( 1+\left\| {\bf{v}}^{n-1}\right\|_{\mathcal{H}^{k}(\Omega)}^{5}\right) 
	\left( 1+\left\| \vn\right\|_{\mathcal{H}^{k}(\Omega)}\right).
	\end{split}
	\label{3.13}
\end{align}
Plugging the  estimates \eqref{3.8} and \eqref{3.11}-\eqref{3.13} of three parts into \eqref{3.7}, we deduce that
\begin{align}
\begin{split}
-K_{3} &\leq 
C\left( \left\| \vn\right\|_{\mathcal{H}^{k}(\Omega)}+\left\| \partial_{y}\vn\right\|_{\mathcal{H}^{k}(\Omega)}\right) 
\left( 1+\left\| {\bf{v}}^{n-1}\right\|_{\mathcal{H}^{k}(\Omega)}^{5}\right) 
\left( 1+\left\| \vn\right\|_{\mathcal{H}^{k}(\Omega)}\right)\\
&
\qquad -\int_{\mathbb{T}} \partial^{\alpha}\vn~ {{\bf{F}}({\bf{v}}_{1}^{n-1}, \partial^{\alpha}{\bf{v}}_{1}^{n-1})} \partial_{y}\vn_{1} \big|_{y=0}dx.
\end{split}
\label{k3}
\end{align}

\textbf{Dealing with $K_4$  term :}
We first get, by using integration by parts, that
\begin{align*}
\begin{split}
K_{4}&=-\left\langle \partial^{\alpha}\vn, \bn\partial^{\alpha}\partial_{y}^{2}\vn \right\rangle \\
&=\int_{\mathbb{T}} \partial^{\alpha}\vn\bn\partial^{\alpha}\partial_{y}\vn \big|_{y=0}~dx
+\left\langle \partial^{\alpha}\partial_{y}\vn, \bn\partial^{\alpha}\partial_{y}\vn \right\rangle
+\left\langle \partial^{\alpha}\vn, \partial_{y}\bn\partial^{\alpha}\partial_{y}\vn \right\rangle.
\end{split}
\end{align*}
Since the matrix ${\bf{B}}$ given in \eqref{B} is positive, it follows form  positive lower bound condition $\eqref{3.2}_{4}$ that 
\begin{align}
	\begin{split}
		-\left\langle \partial^{\alpha}\vn, \bn\partial^{\alpha}\partial_{y}^{2}\vn \right\rangle 
		&\geq \int_{\mathbb{T}} \partial^{\alpha}\vn\bn\partial^{\alpha}\partial_{y}\vn \big|_{y=0}~dx
		\\
		&\quad+c_{\delta}\left\| \partial^{\alpha}\partial_{y}\vn\right\|_{L^2(\Omega)}^{2}-
		C\left\| \vn\right\|_{\mathcal{H}^{k}(\Omega)}
		\left( 1+\left\| {\bf{v}}^{n-1}\right\|_{\mathcal{H}^{3}(\Omega)}^{3}\right) 
	\left\| \partial^{\alpha}\partial_{y }\vn\right\|_{\mathcal{H}^{k}(\Omega)},
	\end{split}
\label{3.15}
\end{align}
where  in the sequel, $c_\delta>0$ depends only on $\delta$, $\mu$, $\mu'$, $\sigma$ and the outflow.

Next, we now turn to handle the boundary layer term in \eqref{3.15}. It is straightforward to observe that the boundary term at 
$y=0$ will vanish in the case where $\partial^{\alpha}=\partial^{\alpha_1}_{t}\partial^{\alpha_2}_{x}$, when performing the estimate
\begin{align}
\int_{\mathbb{T}} \partial^{\alpha}\vn\bn\partial^{\alpha}\partial_{y}\vn \big|_{y=0}~dx
=
\int_{\mathbb{T}} (0, 0, \partial^{\alpha} q^{n})\bn(\partial^{\alpha}\partial_{y}u^{n}, \partial^{\alpha}\partial_{y}w^{n}, 0) \big|_{y=0}~dx=0.
\label{3.16}
\end{align}
Then, we only need to deal the case of $\partial^{\alpha}=\partial^{\beta}\partial_{y}$, $|\beta|\leq k-1$.
The main observation is that we can rewrite the boundary term by the identity \eqref{3.3} as follows:
\begin{align}
&\int_{\mathbb{T}} \partial^{\alpha}\vn\bn\partial^{\alpha}\partial_{y}\vn \big|_{y=0}~dx
\nonumber\\
&
=\int_{\mathbb{T}} \partial^{\beta}\partial_{y}\vn\bn\partial^{\beta}\partial_{y}^{2}\vn \big|_{y=0}~ dx
\nonumber\\
&
=
\int_{\mathbb{T}} \partial^{\beta}\partial_{y}\vn
\partial^{\beta}\left({\bf{S}}({\bf{v}}^{n-1}_{1})\partial_{t}{\bf{v}}^{n} \right) \big|_{y=0} dx
+
\int_{\mathbb{T}} \partial^{\beta}\partial_{y}\vn
\partial^{\beta}\left(\an\partial_{x}\vn \right) \big|_{y=0} dx
\nonumber\\
&\quad
+
\int_{\mathbb{T}} \partial^{\beta}\partial_{y}\vn
\fn\partial^{\beta}\partial_{y}\vn_{1}  \big|_{y=0} dx
+
\int_{\mathbb{T}} \partial^{\beta}\partial_{y}\vn
\left[ \partial^{\beta}, \fn\right] \partial_{y}\vn_{1}  \big|_{y=0} dx
\nonumber\\
&
\quad
-
\int_{\mathbb{T}} \partial^{\beta}\partial_{y}\vn
\left[ \partial^{\beta}, \bn\right]\partial_{y}^{2}\vn \big|_{y=0} dx-
\int_{\mathbb{T}} \partial^{\beta}\partial_{y}\vn
\partial^{\beta}\left(g(\vn) \right) \big|_{y=0} dx=\sum_{i=1}^{6}I_{i}.
\end{align}
By applying the trace estimate \eqref{trace1}, Lemma \ref{sobolevtype} and \eqref{inq3}, we get
\begin{align}
I_{1}&=\int_{\mathbb{T}} \partial^{\beta}\partial_{y}\vn
{\bf{S}}({\bf{v}}^{n-1}_{1})\partial^{\beta}\partial_{t}{\bf{v}}^{n}  \big|_{y=0} dx+\int_{\mathbb{T}} \partial^{\beta}\partial_{y}\vn
\left[\partial^{\beta}, {\bf{S}}({\bf{v}}^{n-1}_{1})\right] \partial_{t}{\bf{v}}^{n}  \big|_{y=0} dx \nonumber\\
& \leq C\left( \left\| \partial^{\beta}\partial\vn|_{y=0}\right\|_{L^2(\mathbb{T})}^{2}\left\| \sn|_{y=0}\right\|_{L^\infty(\mathbb{T})}
+
\left\| \partial^{\beta}\partial_{y}\vn\right\|_{L^2(\Omega)}^{\frac{1}{2}}\left\| \partial^{\beta}\partial_{y}^{2}\vn\right\|
_{L^2(\Omega)}^{\frac{1}{2}}
\left\| \left[\partial^{\beta}, {\bf{S}}({\bf{v}}^{n-1}_{1})\right] \partial_{t}{\bf{v}}^{n}\right\|_{\mathcal{H}^1(\Omega)}
\right) 
\nonumber\\
& \leq C\left\| \partial_{y}\vn\right\|_{\mathcal{H}^k(\Omega)}
 \left\|\vn\right\|_{\mathcal{H}^k(\Omega)}\left\| \sn\right\|_{H^3(\Omega)}+
\left( \left\|\vn\right\|_{\mathcal{H}^k(\Omega)}+ \left\| \partial_{y}\vn\right\|_{\mathcal{H}^k(\Omega)}\right) 
\left\| \sn\right\|_{\mathcal{H}^k(\Omega)}\left\|\partial_{ {t}}\vn\right\|_{\mathcal{H}^{k-1}(\Omega)}
\nonumber\\
& \leq C\left( \left\|\vn\right\|_{\mathcal{H}^k(\Omega)}+ \left\| \partial_{y}\vn\right\|_{\mathcal{H}^k(\Omega)}\right) \left\|\vn\right\|_{\mathcal{H}^k(\Omega)}
\left( 1+\left\| \vnn\right\|_{\mathcal{H}^k(\Omega)}^4
\right) \label{i1} .
\end{align}
Similarly, we have
\begin{align}
	I_{2} \leq C\left( \left\|\vn\right\|_{\mathcal{H}^k(\Omega)}+ \left\| \partial_{y}\vn\right\|_{\mathcal{H}^k(\Omega)}\right)\left\|\vn\right\|_{\mathcal{H}^k(\Omega)}
	\left( 1+\left\| \vnn\right\|_{\mathcal{H}^k(\Omega)}^5
	\right) .
\label{i2}
\end{align}
and
\begin{align}
	I_{3} \leq C\left\| \partial_{y}\vn\right\|_{\mathcal{H}^k(\Omega)}\left\|\vn\right\|_{\mathcal{H}^k(\Omega)}
	\left( 1+\left\| \vnn\right\|_{\mathcal{H}^k(\Omega)}^5
	\right) .
	\label{i3}
\end{align}
Notice that
\begin{align*}
\begin{split}
	I_{4}
	&=
	\int_{\mathbb{T}} \partial^{\beta}\partial_{y}\vn	{{\bf{F}}({\bf{v}}_{1}^{n-1},  \partial^{\beta} \partial_{y}{\bf{v}}_{1}^{n-1})}
	 \partial_{y}\vn_{1}  \big|_{y=0} dx  \\
	&\quad+
	\sum\limits_{ \beta' \leq \beta, 1\leq |\beta'|}
	\sum\limits_{ 0\leq \gamma' \leq \beta',  |\gamma'| \leq |\beta'|-1}
	C^{\beta'}_\beta C^{\gamma'}_{\beta'}
\int_{\mathbb{T}} \partial^{\beta}\vn, 
	{{\bf{F}}(\partial^{\beta'-\gamma'}{\bf{v}}_{1}^{n-1}, \partial^{\gamma'}\partial_{y}{\bf{v}}_{1}^{n-1})}
	\partial^{\beta-\beta'}\partial_{y}\vn_{1} \big|_{y=0}dx .
\end{split}
\end{align*}
It is obvious that the structure of
$I_{4}$ are similar to the last two terms on the right-hand side of \eqref{3.7}. Then, in a similar way with $K_{3}$, we can infer, by case analysis, that 
\begin{align*}
&		\sum\limits_{ \beta' \leq \beta, 1\leq |\beta'|}
		\sum\limits_{ 0\leq \gamma' \leq \beta',  |\gamma'| \leq |\beta'|-1}
		C^{\beta'}_\beta C^{\gamma'}_{\beta'}
		\int_{\mathbb{T}} \partial^{\beta}\vn, 
		{{\bf{F}}(\partial^{\beta'-\gamma'}{\bf{v}}_{1}^{n-1}, \partial^{\gamma'}\partial_{y}{\bf{v}}_{1}^{n-1})}
		\partial^{\beta-\beta'}\partial_{y}\vn_{1} \big|_{y=0}dx 
	\\
& \leq 
C\left\| \partial_{y}\vn\right\|_{\mathcal{H}^k(\Omega)}\left\|\vn\right\|_{\mathcal{H}^k(\Omega)}
\left( 1+\left\| \vnn\right\|_{\mathcal{H}^k(\Omega)}^5
\right) ,
\end{align*}
which directly gives 
\begin{align}
		I_{4}
	\leq 
		\int_{\mathbb{T}} \partial^{\alpha}\vn	{{\bf{F}}({\bf{v}}_{1}^{n-1},  \partial^{\alpha} {\bf{v}}_{1}^{n-1})}
		\partial_{y}\vn_{1}  \big|_{y=0} dx  
		+
C\left\| \partial_{y}\vn\right\|_{\mathcal{H}^k(\Omega)}\left\|\vn\right\|_{\mathcal{H}^k(\Omega)}
\left( 1+\left\| \vnn\right\|_{\mathcal{H}^k(\Omega)}^5
\right) .
\label{i4}
\end{align}
In addition, it directly follows 
\begin{align*}
I_{5} \leq  \left\| \partial^{\beta}\partial_{y}\vn|_{y=0}\right\|_{L^2(\mathbb{T}_{x})}^{2}\left\| \left[ \partial^{\beta}, \bn\right]\partial_{y}^{2}\vn|_{y=0}\right\|_{L^2(\mathbb{T}_{x})}.
\end{align*}
Again due to \eqref{trace1}, Lemma \ref{sobolevtype} and \eqref{inq3}, we derive that
\begin{align*}
&\left\| \left[ \partial^{\beta}, \bn\right]\partial_{y}^{2}\vn |_{y=0} \right\|_{L^2({\mathbb{T}_{x}})}
\\
&\leq 
\left\|  \partial^{\beta} \bn |_{y=0} \right\|_{L^2({\mathbb{T}_{x}})}
\left\| \partial_{y}^{2}\vn |_{y=0} \right\|_{L^\infty({\mathbb{T}_{x}})}
+\sum\limits_{  1\leq |\beta'|\leq |\beta|-1}
C^{\beta'}_\beta 
\left\|  \partial^{\beta'} \bn |_{y=0} \right\|_{L^\infty({\mathbb{T}_{x}})}
\left\| \partial^{\beta-\beta'}\partial_{y}^{2}\vn |_{y=0}
 \right\|_{L^2({\mathbb{T}_{x}})}
 \\
&\leq
C\left( 1+\left\| \vnn\right\|_{\mathcal{H}^{k}(\Omega)}^{3}\right) 
\left( \left\| \partial^{2}_{y}\vn\right\|_{H^{2}(\Omega)}
+\left\| \partial^{2}_{y}\vn\right\|_{\mathcal{H}^{k-2}(\Omega)}^{\frac12}
\left\| \partial^{3}_{y}\vn\right\|_{\mathcal{H}^{k-2}(\Omega)}^{\frac12}
\right) \\
&
\leq 
C\left( 1+\left\| \vnn\right\|_{\mathcal{H}^{k}(\Omega)}^{3}\right) 
\left( \left\| \vn\right\|_{\mathcal{H}^{k}(\Omega)}
+\left\| \vn\right\|_{\mathcal{H}^{k}(\Omega)}^{\frac12}
\left\| \partial_{y}\vn\right\|_{\mathcal{H}^{k}(\Omega)}^{\frac12}
\right),
\end{align*}
and then, 
\begin{align}
	I_{5}
	&\leq 
	C\left\| \partial^{\beta}\partial_{y}\vn\right\|_{L^2(\Omega)}^{\frac{1}{2}}\left\| \partial^{\beta}\partial_{y}^{2}\vn\right\|
	_{L^2(\Omega)}^{\frac{1}{2}}
	\left( 1+\left\| \vnn\right\|_{\mathcal{H}^{k}(\Omega)}^{2}\right) 
	\left( \left\| \vn\right\|_{\mathcal{H}^{k}(\Omega)}
	+\left\| \vn\right\|_{\mathcal{H}^{k}(\Omega)}^{\frac12}
	\left\| \partial_{y}\vn\right\|_{\mathcal{H}^{k}(\Omega)}^{\frac12}
	\right)
	\nonumber\\
&\leq 
 C\left( \left\|\vn\right\|_{\mathcal{H}^k(\Omega)}+ \left\| \partial_{y}\vn\right\|_{\mathcal{H}^k(\Omega)}\right)\left\|\vn\right\|_{\mathcal{H}^k(\Omega)}
\left( 1+\left\| \vnn\right\|_{\mathcal{H}^k(\Omega)}^3
\right).\label{i5}
\end{align}
Finally, a similar derivation with $I_{5}$ gives rise to
\begin{align}
I_{6}&\leq
	C\left\| \partial^{\beta}\partial_{y}\vn\right\|_{L^2(\Omega)}^{\frac{1}{2}}\left\| \partial^{\beta}\partial_{y}^{2}\vn\right\|
_{L^2(\Omega)}^{\frac{1}{2}}
\left\| \partial^{\beta}\left(g(\vnn) \right)\right\|_{L^2(\Omega)}^{\frac{1}{2}}\left\| \partial^{\alpha}\left(g(\vnn) \right)\right\|
_{L^2(\Omega)}^{\frac{1}{2}}
\nonumber\\
& \leq C\left\| \vn\right\|_{\mathcal{H}^{k}(\Omega)}^{\frac12}
\left\| \partial_{y}\vn\right\|_{\mathcal{H}^{k}(\Omega)}^{\frac12}
\left( 1+\left\| \vnn\right\|_{\mathcal{H}^{k-1}(\Omega)}^4\right)^{\frac12}
\left( 1+\left\| \vnn\right\|_{\mathcal{H}^{k}(\Omega)}^4\right)^{\frac12}
\nonumber\\
&\leq 
C\left( \left\|\vn\right\|_{\mathcal{H}^k(\Omega)}+ \left\| \partial_{y}\vn\right\|_{\mathcal{H}^k(\Omega)}\right)
\left( 1+\left\| \vnn\right\|_{\mathcal{H}^k(\Omega)}^4
\right).
\label{i6}
\end{align}
Combining estimates \eqref{i1}-\eqref{i6} of $I_{1}-I_{6}$, we have
\begin{align*}
-&\int_{\mathbb{T}} \partial^{\alpha}\vn\bn\partial^{\alpha}\partial_{y}\vn \big|_{y=0}~dx
\\
&\geq
	-\int_{\mathbb{T}} \partial^{\alpha}\vn	{{\bf{F}}({\bf{v}}_{1}^{n-1},  \partial^{\alpha} {\bf{v}}_{1}^{n-1})}
\partial_{y}\vn_{1}  \big|_{y=0} dx  
-C\left( \left\|\vn\right\|_{\mathcal{H}^k(\Omega)}+ \left\| \partial_{y}\vn\right\|_{\mathcal{H}^k(\Omega)}\right)\left\|\vn\right\|_{\mathcal{H}^k(\Omega)}
\left( 1+\left\| \vnn\right\|_{\mathcal{H}^k(\Omega)}^5
\right)
\end{align*}
from which and \eqref{3.15}, it follows that
\begin{align}
	\begin{split}
	K_{4}=	-\left\langle \partial^{\alpha}\vn, \bn\partial^{\alpha}\partial_{y}^{2}\vn \right\rangle 
		&\geq -\int_{\mathbb{T}} \partial^{\alpha}\vn	{{\bf{F}}({\bf{v}}_{1}^{n-1},  \partial^{\alpha} {\bf{v}}_{1}^{n-1})}
		\partial_{y}\vn_{1}  \big|_{y=0} dx  +c_{\delta}\left\| \partial^{\alpha}\partial_{y}\vn\right\|_{L^2(\Omega)}^{2}
		\\
		&\qquad
		-C\left( \left\|\vn\right\|_{\mathcal{H}^k(\Omega)}+ \left\| \partial_{y}\vn\right\|_{\mathcal{H}^k(\Omega)}\right)\left\|\vn\right\|_{\mathcal{H}^k(\Omega)}
		\left( 1+\left\| \vnn\right\|_{\mathcal{H}^k(\Omega)}^5
		\right).
	\end{split}
	\label{k4}
\end{align}
\textbf{Dealing with $K_5$-$K_6$  terms :} Simple and direct calculations give
\begin{align}
K_{5}+K_{6} & \leq 
+
\left| \left\langle \partial^{\alpha}\vn, \left[ \partial^{\alpha}, \bn\right] \partial_{y}^{2}\vn	\right\rangle\right| +\left| \left\langle \partial^{\alpha}\vn, \partial^{\alpha}g(\vn)\right\rangle \right| 
\nonumber\\
&
\leq C\left\|\vn\right\|_{\mathcal{H}^k(\Omega)} 
\left( 1+\left\| \vnn\right\|_{\mathcal{H}^k(\Omega)}^4
\right)
\left( 1+\left\|\vn\right\|_{\mathcal{H}^k(\Omega)}+ \left\| \partial_{y}\vn\right\|_{\mathcal{H}^k(\Omega)}\right).
\label{k56}
\end{align}

Substituting \eqref{k1}, \eqref{k2}, \eqref{k3},  \eqref{k4} and \eqref{k56} into \eqref{3.3}, it follows from Young's inequality that for $|\alpha|\leq k, k\geq 4$, 
\begin{align}
&\frac{1}{2}\frac{d}{dt}\sum\limits_{|\alpha| \leq k}\left\langle \partial^{\alpha}\vn, \sn \partial^{\alpha}\vn\right\rangle
+c_{\delta}\left\| \partial_{y}\vn\right\|_{\mathcal{H}^{k}(\Omega)}^{2}
\nonumber\\
&\leq
C\left( \left\|\vn\right\|_{\mathcal{H}^k(\Omega)}+ \left\| \partial_{y}\vn\right\|_{\mathcal{H}^k(\Omega)}\right)
\left( 1+\left\|\vn\right\|_{\mathcal{H}^k(\Omega)}\right) 
\left( 1+\left\| \vnn\right\|_{\mathcal{H}^k(\Omega)}^5
\right)
\nonumber\\
&
\leq
\frac{c_{\delta}}{2}\left\| \partial_{y}\vn\right\|_{\mathcal{H}^{k}(\Omega)}^{2}+
C\left( 1+\left\|\vn\right\|_{\mathcal{H}^k(\Omega)}^2\right) 
\left( 1+\left\| \vnn\right\|_{\mathcal{H}^k(\Omega)}^{10}
\right).
\label{3.17}
\end{align}
Since the matrix ${\bf{S}}$ is positive  definite,  it follows, by $\min\{\mathcal{A}^{-1}(P-q_{1})q_{1}, \frac{ Q}{ 2\gamma}\} \geq c'_{\delta,\gamma}$ due to positive lower bound 
assumption $\eqref{3.2}_{4}$, that
\begin{align*}
\left\langle \partial^{\alpha}\vn, \sn \partial^{\alpha}\vn\right\rangle
\geq c'_{\delta,\gamma} \left\| \partial^{\alpha}\vn\right\|_{L^2(\Omega)},
\end{align*}
which concludes that
\begin{align*}
\left\|\vn\right\|_{\mathcal{H}^k(\Omega)}^2 \leq C \sum\limits_{|\alpha| \leq k}\left\langle \partial^{\alpha}\vn, \sn \partial^{\alpha}\vn\right\rangle.
\end{align*}
Putting \eqref{3.17} together, we obtain
\begin{align*}
	&\frac{d}{dt}\sum\limits_{|\alpha| \leq k}\left\langle \partial^{\alpha}\vn, \sn \partial^{\alpha}\vn\right\rangle
	+\left\| \partial_{y}\vn\right\|_{\mathcal{H}^{k}(\Omega)}^{2}
	\\
	&
	\leq
	C\left( 1+\sum\limits_{|\alpha| \leq k}\left\langle \partial^{\alpha}\vn, \sn \partial^{\alpha}\vn\right\rangle\right) 
	\left( 1+\left\| \vnn\right\|_{\mathcal{H}^k(\Omega)}^{10}
	\right).
\end{align*}
Summing up, we derive from Gronwall's inequality 
  that
\begin{align}
\begin{split}
	&\sum\limits_{|\alpha| \leq k}\left\langle \partial^{\alpha}\vn, \sn \partial^{\alpha}\vn\right\rangle(t)
+\int_{0}^{t}\left\| \partial_{y}\vn\right\|_{\mathcal{H}^{k}(\Omega)}^{2} d\tau
\\
&
\leq
C\left( 1+\sum\limits_{|\alpha| \leq k}\left\langle \partial^{\alpha}\vn, \sn \partial^{\alpha}\vn\right\rangle\big|_{t=0}\right) \cdot
\exp \left( C\int_{0}^{t}\left\| \vnn\right\|_{\mathcal{H}^k(\Omega)}^{10} ~d\tau
\right),
\end{split}
\end{align}
from which, initial condition \eqref{initial} for approximate $n$-th order solution and  \eqref{m13}, 
we infer that
\begin{align*}
	&\left\|\vn\right\|_{\mathcal{H}^k(\Omega_{t})}^2
+\int_{0}^{t}\left\| \partial_{y}\vn\right\|_{\mathcal{H}^{k}(\Omega)}^{2} d\tau
\leq
C_{0}(T_0, M_e, \delta, \gamma)M_{1}^{3}\exp \left( C_{0}(T_0, M_e, \delta, \gamma)\int_{0}^{t}\left\| \vnn\right\|_{\mathcal{H}^k(\Omega)}^{10} ~d\tau
\right).
\end{align*}
Hence, the proof of Lemma \ref{solvability} is completed.

\end{proof}

Now, we  need to verify that the approximate solution $\vn$  to initial-boundary value problem \eqref{3.00} is uniform bounded and satisfies the corresponding 
$n$-th order induction hypothesis $\eqref{3.2}$. To be specific
we have the following lemma.

\begin{lemma}\label{bounded}
Under the assumptions of Theorem \ref{local}, 
the approximate solution
sequence $\{\vn\}_{n \geq 0}=\{(u^{n}, w^{n}, q^{n})\}_{n \geq 0}$
 constructed in the  subsection \ref{s3.1} satisfies that
 \begin{align}
\left\| \vn\right\|_{\mathcal{H}^{k}(\Omega_{T_{1}})} \leq M, \qquad \forall n \geq 0,
\label{4.1}
 \end{align}
 and
  \begin{align}
 \delta \leq q^{n}(t,x,y)+ H^{2}(t,x)/2\leq P(t,x)-\delta, \qquad (t, x, y) \in \Omega_{T_{1}},
 	\label{4.2}
 \end{align}
 for any $ T_1\in(0,T_0]$.
\end{lemma}

\begin{proof}

We prove Lemma \ref{bounded} by induction on $n$. Firstly, 
fix $T'=\min\left\lbrace T_{0}, \frac{1}{5C_{0}^{6}M_{1}^{15}}\right\rbrace $
and assume that  for $n \geq 1$,
\begin{align}
\left\|\vnn\right\|_{\mathcal{H}^k(\Omega_{t})}^2 \leq 
\frac{C_{0}M_{1}^3}{\left(1-5C_{0}^{6}M_{1}^{15}t\right)^{\frac15} },
\qquad \forall t \in [0, T').
\label{4.3}
\end{align}
It is straightforward to verify that the zero-th order approximate solution ${\bf{v}}^{0}$ satisfies \eqref{4.3} provided the constant $C_0$ is chosen sufficiently large.
Therefore, by the induction  of all  $n\geq 0$ and \eqref{g}, it holds that 
\begin{align}
\left\|\vn\right\|_{\mathcal{H}^k(\Omega_{t})}^2 \leq C_{0}M_{1}^{3}\exp \left( \int_{0}^{t}\frac{C_{0}^{6}M_{1}^{15}}{1-5C_{0}^{6}M_{1}^{15}\tau } ~d\tau
\right) \leq \frac{C_{0}M_{1}^3}{\left(1-5C_{0}^{6}M_{1}^{15}t\right)^{\frac15} },
\label{4.4}
\end{align}
for $t\in[0, T')$.
Moreover, due to \eqref{assumptions0}) , we have
\begin{align*}
q^{n}(0,x,y) + H^2(0,x)/2 = h_{1,0}^{2}(x,y)/2 \geq 2\delta,
\end{align*}
from which, \eqref{me} and  the Sobolev embedding inequality,  we achieve
\begin{align*}
&q^{n}(t,x,y) + H^{2}(t,x) /2= q^{n}(0,x,y) + H^2(0,x)/2+\int_0^t \left(\partial_t q^{n}(\tau,x,y) + \partial_t \left( H^{2}(\tau,x)/2\right) \right) d\tau
\\
&\geq 2\delta - \int_0^t \left(\|\partial_t q^{n}(\tau,\cdot)\|_{L^\infty(\Omega)} + \|\partial_t (H^2(\tau,\cdot)/2)\|_{L^\infty(\mathbb{T}_{x})}\right) d\tau
\\
&
\geq 2\delta - t\left(\|q^{n}(\tau)\|_{\mathcal{H}^3(\Omega_{t})} + \| H(\tau,\cdot)\|_{H^2(\mathbb{T}_x)}^2\right) \\
&\geq 2\delta - 2t\sqrt{C_{0}M_{1}^3}{\left(1-3C_{0}^{6}M_{1}^{15}t\right)^{-\frac{1}{10}} }.
\end{align*}
Similar estimate holds for  $P(t,x)-q^{n}(t,x,y)- H^{2}(t,x)/2$.
Let $T_{1}$  be chosen such that 
\begin{align*}
2T_{1}\sqrt{C_{0}M_{1}^3}{\left(1-3C_{0}^{6}M_{1}^{15}T_{1}\right)^{-\frac{1}{10}} } \leq \delta, 
\qquad \forall ~ T_{1} \in (0, T').
\end{align*}
Consequently,  by \eqref{4.4}, we have 
\begin{align}
M=\frac{\sqrt{C_{0}M_{1}^3}}{\left(1-5C_{0}^{6}M_{1}^{15}T_{1}\right)^{\frac{1}{10}} }.
\label{M}
\end{align}
This completes the proof of Lemma \ref{bounded}.

\end{proof}

\subsection{Compactness property}
In this subsection, we will the compactness
of sequence  $\{\vn\}_{n \geq 0}=\{(u^{n}, w^{n}, q^{n})\}_{n \geq 0}$, which is contractive in $L^2$.
\begin{lemma}\label{contraction}
	
	Let $T_{*}=\min\left\lbrace T_{1}, T_{2}\right\rbrace $ with $T_{2}$  given by
$$T_{2}= \frac{1}{2C_{1}e^{C_{1}T_{1}}},$$
 then for any $t\in[0, T_*]$, it holds that
\begin{align}
	\sup_{0\leq \tau \leq t}\left\|  {\bf{V}}^{n}(\tau)\right\| _{L^2(\Omega)}^{2}
	\leq\frac{1}{2}\sup_{0\leq \tau \leq t}\left\|  {\bf{V}}^{n-1}(\tau)\right\| _{L^2(\Omega)}^{2}, \quad n\geq1.
\label{4.6}
\end{align}
\end{lemma}
\begin{proof}
Setting
\begin{align*}
{\bf{V}}^{n}={\bf{v}}^{n+1}-{\bf{v}}^{n}=(u^{n+1}, w^{n+1}, q^{n+1})
-(u^{n}, w^{n}, q^{n}), \qquad n\geq 1,
\end{align*}
then based on \eqref{3.00}, ${\bf{V}}^{n}$ satisfies
\begin{equation}\label{4.5}
	\left\{
	\begin{aligned}
		&{\bf{S}}({\bf{v}}^{n}_{1})\partial_{t}{\bf{V}}^{n}
		+{\bf{A}}(\vn_{1})\partial_{x}{\bf{V}}^{n}
		+{{\bf{F}}({\bf{v}}_{1}^{n}, \partial_{y}{\bf{v}}_{1}^{n})} \partial_{y}{\bf{V}}^{n}
		-{\bf{B}}(\vn_{1})\partial_{y}^{2}{\bf{V}}^{n} \\
		&\quad=\left(g(\vn_{1} -g(\vnn_{1})\right) 
		-\left({\bf{S}}(\vn_{1})-{\bf{S}}(\vnn_{1}) \right)
		\partial_{t}\vn
		-\left({\bf{A}}(\vn_{1})-{\bf{A}}(\vnn_{1}) \right) \partial_{x}\vn
		\\
		&\qquad
		-
		\left( {{\bf{F}}({\bf{v}}_{1}^{n}, \partial_{y}{\bf{v}}_{1}^{n})}-
		{{\bf{F}}({\bf{v}}_{1}^{n-1}, \partial_{y}{\bf{v}}_{1}^{n-1})}\right) \partial_{y}\vn
		+\left({\bf{B}}(\vn_{1})-{\bf{B}}(\vnn_{1}) \right) \partial_{y}^{2}\vn
		,\\
		&(u^{n+1}-u^{n}, w^{n+1}-w^{n}, \partial_{y}(q^{n+1}-q^{n})\big|_{y=0}=\mathbf{0},
		\quad 
		\lim_{ {y}\rightarrow +\infty}{\bf{V}}^{n}( {t}, {x}, {y})=\mathbf{0},\quad {\bf{V}}^{n}|_{ {t}=0}=\mathbf{0}.
	\end{aligned}
	\right.
\end{equation}
Similar to the analysis of a priori estimate \eqref{3.17} in Lemma \ref{solvability}, we can arrive at 
\begin{align*}
	&\frac{1}{2}\frac{d}{dt}\sum\limits_{|\alpha| \leq k}\left\langle {\bf{V}}^{n}, {\bf{S}}({\bf{v}}^{n}_{1}) {\bf{V}}^{n}\right\rangle
	+\left\| \partial_{y}{\bf{V}}^{n}\right\|_{{L}^{2}(\Omega)}^{2}
	\nonumber\\
	&\leq
	C\mathcal{P}\left( \left\|\vn\right\|_{\mathcal{H}^4(\Omega)}, \left\| \vnn\right\|_{\mathcal{H}^4(\Omega)}\right)
	\left( \left\|{\bf{V}}^{n-1}\right\|_{L^2(\Omega)}^{2}+ \left\|{\bf{V}}^{n}\right\|_{L^2(\Omega)}^{2}\right) 
,
\end{align*}
where $\mathcal{P}(\cdot, \cdot)$ is a generic polynomial functional.
Using the uniform estimate \eqref{4.1} and positive lower bound condition $\eqref{3.2}_{4}$, we infer that there exist  a constant $C_{1}$ only depending on $T_1, M, \delta, \gamma$ such that the following estimate holds :
\begin{align*}
	\frac{d}{dt}\left\| {\bf{V}}^{n}\right\|_{{L}^{2}(\Omega)}^{2}
	+\left\| \partial_{y}{\bf{V}}^{n}\right\|_{{L}^{2}(\Omega)}^{2}
	\leq
	C_{1}
	\left( \left\|{\bf{V}}^{n-1}\right\|_{L^2(\Omega)}^{2}+ \left\|{\bf{V}}^{n}\right\|_{L^2(\Omega)}^{2}\right) 
	,
\end{align*}
which, together with Gronwall's lemma, yields directly
\begin{align*}
\sup_{0\leq \tau \leq t}\left\|  {\bf{V}}^{n}(\tau)\right\| _{L^2(\Omega)}^{2} \leq C_{1}e^{C_{1}t}\int_{0}^{t} \left\|  {\bf{V}}^{n-1}(\tau)\right\| _{L^2(\Omega)}^{2} d\tau \leq
C_{1}te^{C_{1}t} \sup_{0\leq \tau \leq t}\left\|  {\bf{V}}^{n-1}(\tau)\right\| _{L^2(\Omega)}^{2}, \forall t \in [0, T_{1}].
\end{align*}
Choosing $T_{2}= \frac{1}{2C_{1}e^{C_{1}T_{1}}}$, then we obtain
\begin{align*}
	\sup_{0\leq \tau \leq t}\left\|  {\bf{V}}^{n}(\tau)\right\| _{L^2(\Omega)}^{2}
	\leq\frac{1}{2}\sup_{0\leq \tau \leq t}\left\|  {\bf{V}}^{n-1}(\tau)\right\| _{L^2(\Omega)}^{2}, \forall t \in [0, T_{*}],
\end{align*}
where $T_{*}=\min\left\lbrace T_{1}, T_{2}\right\rbrace $.
Therefore, we complete the proof of this lemma.
\end{proof}

\subsection{Proof of Theorem \ref{local}}
In this subsection, we will finish the proof of Theorem \ref{local}. Indeed,
by virtue of Lemmas \ref{bounded}-\ref{contraction}, it can be shown that the approximation sequence $\{\vn\}_{n\geq 0}$ forms a Cauchy sequence in  $\mathcal{H}^{k'} (\Omega_{T_{*}})$ with $k' < k$.  Consequently, the sequence $\{\vn\}_{n\geq 0}$ converges strongly in $\mathcal{H}^{k'} (\Omega_{T_{*}})$.  Therefore, there exists ${\bf{v}}=(u, w, q)^{\it T} \in \mathcal{H}^{k'} (\Omega_{T_{*}})$ such that 
\begin{align*}
\lim_{n\rightarrow +\infty}\vn={\bf{v}} \quad \mathrm{in} \quad \mathcal{H}^{k'} (\Omega_{T_{*}}),
\end{align*}
with the norm bound
\begin{align*}
	\left\| {\bf{v}}\right\|_{\mathcal{H}^{k}(\Omega_{T_{1}})} \leq M, \qquad \forall n \geq 0,
\end{align*}
where the constant $M$ is given in \eqref{M}.
By passing to the limit $n \rightarrow \infty $ in \eqref{3.00} and utilizing the decompositions
\begin{align*}
	u_{1}=u+U(t,x)\phi(y), \qquad w_{1}=w+I(t,x), \qquad  q_{1}=q+H^{2}(t,x)/2,
\end{align*}
we can conclude that ${\bf{v}}_{1}=(u_{1}, w_{1}, q_{1})^{\it T}$ is a classical solution to the problem \eqref{3m}. The proof of Theorem \ref{local} is completed.

\section{Existence for the original problem}\label{s4}
In this section, we prove that the solution  $(\hat{u}_1, \hat{w}_{1},\hat{h}_1)$ of the system \eqref{hat}  can be transformed to a classical solution $({u}_1, {w}_{1},{h}_1)$ of the original  problem
\eqref{1 mm boundary layer}.

Firstly, it follows from
Theorem \ref{local} that there exist a time $T_* \in (0,T]$ such that initial-boundary value problem \eqref{hat} has a unique classical solution
$(\hat{u}_1,\hat{w}_{1},\hat{h}_1)(t,x,\bar{y})$  satisfying
\begin{align}
\hat{h}_{1}(t, x, \bar{y})\geq \delta, \qquad	\frac{1}{2}\hat{h}_{1}^2(t,x,\bar{y}) \leq P(t,x)-\delta, \qquad\forall (t,x,\bar{y})\in D_{T_*},
\label{asm}
\end{align}
where we have replaced the variable $(\bar{t}, \bar{x})$ by $(t,x)$.

Next, define $\psi=\psi(t,x,y)$ by the following form
\begin{align}
	y=\int^{\psi(t,x,y)}_0\frac{d\bar{y}}{\hat{h}_1(t,x,\bar{y})}.
\label{4.7}
\end{align}
By \eqref{asm}, we deduce that $\psi$ is well-defined and continuous in $D_{T_*}$. It holds
\begin{align}
	\partial_y \psi(t,x,y)=\hat{h}_1\left( t,x,\psi(t,x,y)\right),
\label{4.8}
\end{align}
which,  together with the upper and lower bounds of $\hat{h}_1$ given in \eqref{asm}, yields that 
\begin{align}
	\psi|_{y=0}=0,\quad \psi|_{y\rightarrow+\infty}\rightarrow+\infty,
\label{4.9}
\end{align}
and
\begin{align}
	\psi(0,x,y)=\bar{y}(x,y).
\label{4.10}
\end{align}
In addition , simple and direct calculations  by using \eqref{4.7} show
\begin{equation}
	\left\{\begin{aligned}
		&\partial_t \psi(t,x,y)
		=\hat{h}_1\left( t,x,\psi(t,x,y)\right) \int^{\psi(t,x,y)}_0\frac{\partial_t\hat{h}_1}{\hat{h}_{1}^2}(t,x,\bar{y})d\bar{y},
		\\
		&\partial_x \psi(t,x,y)
		=\hat{h}_1\left( t,x,\psi(t,x,y)\right) \int^{\psi(t,x,y)}_0\frac{\partial_x\hat{h}_1}{\hat{h}_{1}^2}(t,x,\bar{y})d\bar{y},
		\\
		&\partial_{y}^2\psi(t,x,y)=\hat{h}_1\left( t,x,\psi(t,x,y)\right) \partial_{\bar{y}}\hat{h}_1\left( t,x,\psi(t,x,y)\right),
		\\
		&\partial_{y}^3\psi(t,x,y)=\hat{h}_1\left( t,x,\psi(t,x,y)\right) \partial_{\bar{y}}\left( (\hat{h}_1\partial_{\bar{y}}\hat{h}_1) (t,x,\psi(t,x,y))\right).
	\end{aligned}\label{4.11}
	\right.
\end{equation}
By means of $\psi$, we define 
\begin{align}
	(u_{1},w_1,h_{1})=(u_{1},w_{1} ,h_{1}) (t, x, y):= \hat{u}_{1},\hat{w}_{1}, \hat{h}_{1})\left(t,x,\psi(t,x,y)\right) 
\label{4.12}
\end{align}
 and then, set
\begin{align}
	u_2(t,x,y)=-\frac{(\partial_t\psi+u_1\partial_x\psi-\sigma\partial_{y}^2\psi)}{h_1}(t,x,y),\qquad h_2(t,x,y)=-\partial_x\psi(t,x,y).
\end{align}
Thus, $(u_1, u_2, w_1, h_1, h_2)(t,x,y)$ is  a
classical solution of initial-boundary value problem \eqref{1 mm boundary layer}, which   is the desired solution of Theorem \ref{t1}.
The   detailed proof of assertion  is similar to section 3 in \cite{HLY2019}, we omit it for brevity.

Finally, it is easy to check by using \eqref{4.11} that,  $(u_{1}, w_{1}, h_{1}), \partial_y(u_{1}, w_{1}, h_{1})$
and $\partial_y^2(u_{1}, w_{1}, h_{1})$ are continuous and bounded in $D_{T_*}$;
$\partial_t (u_{1}, w_{1}, h_{1})$, $\partial_x(u_{1}, h_{1})$, $(u_{2}, h_{2})$ and $\partial_y(u_{2}, h_{2})$ are continuous and bounded in any compact set of $D_{T_*}$.  The proof of Theorem \ref{t1} is completed.

\begin{appendices}

\end{appendices}

\section*{Statement about conflicting interests}
The authors declare that there are no conflicts of interest.

\section*{Acknowledgements}
Yuming Qin was supported  by the NNSF of China with contract number 12171082, the fundamental
research funds for the central universities with contract numbers 2232022G-13, 2232023G-13, 2232024G-13.

	\phantomsection

\end{document}